\documentclass[numbers,webpdf]{ima-authoring-template}%

\graphicspath{{Fig/}}

\usepackage{mathtools,cleveref,subcaption,booktabs,amsmath,amssymb,braket,dsfont,amsthm}

\theoremstyle{thmstyletwo}%
\newtheorem{theorem}{Theorem}
\newtheorem{proposition}[theorem]{Proposition}%
\newtheorem{remark}[theorem]{Remark}%

\newtheorem{corollary}[theorem]{Corollary}
\newtheorem{assumption}{Assumption}
\newtheorem{lemma}[theorem]{Lemma}

\numberwithin{equation}{section}
\numberwithin{theorem}{section}

\newcommand{\Th}{\mathcal{T}_h}
\newcommand{\V}{\mathbf{V}}
\newcommand{\Vz}{\mathbf{V}_0}
\newcommand{\Vt}{\mathbf{V}_2}
\newcommand{\Vh}{\mathbf{V}_h}
\newcommand{\W}{\mathbf{W}}
\newcommand{\Q}{\mathbf{Q}}
\newcommand{\X}{\mathbf{X}}
\newcommand{\Y}{\mathbf{Y}}

\newcommand\na{\nabla}

\newcommand\Om{\Omega}
\newcommand\pa{\partial}

\newcommand{\lk}{\ln{\frac{T}{k}}}

\newcommand{\norm}[1]{\lVert#1\rVert}

\newcommand{\q}{{\bf{q}}}
\newcommand{\dq}{{\bf{\delta q}}}
\newcommand{\pp}{{\bf{p}}}
\newcommand{\uu}{{\bf{u}}}
\newcommand{\ff}{{\bf{f}}}
\newcommand{\vv}{{\bf{v}}}
\newcommand{\zz}{{\bf{z}}}
\newcommand{\yy}{{\bf{y}}}
\newcommand{\ww}{{\bf{w}}}
\newcommand{\oo}{{\bf{0}}}
\newcommand{\IOprod}[1]{\left(#1\right)_{I\times\Omega}}
\newcommand{\Oprod}[1]{\left(#1\right)_{\Omega}}
\newcommand{\IOpair}[1]{\left\langle#1\right\rangle_{I\times\Omega}}
\newcommand{\Opair}[1]{\left\langle#1\right\rangle_{\Omega}}
\newcommand{\Iprod}[1]{\left(#1\right)_{I}}
\newcommand{\Ipair}[1]{\left\langle#1\right\rangle_{I}}

\newcommand{\Gw}{G_{\bf{w}}}

\newcommand{\R}{\mathbb R}

\renewcommand{\d}{\mathrm{d}}
\newcommand{\NBV}{\mathrm{NBV}}
\newcommand{\BV}{\mathrm{BV}}
\newcommand{\Pad}{P_{[\q_a,\q_b]}}

\DeclareMathOperator{\supp}{supp}

\DeclarePairedDelimiter{\twonorm}{\lVert}{\rVert_{L^2(\Omega)}}

\newcommand{\revB}[1]{\textcolor{black}{#1}}
\newcommand{\revC}[1]{{#1}}
\newcommand{\revD}[1]{{#1}}
\newcommand{\revision}[1]{\textcolor{black}{#1}}
\newcommand{\LtwoLtwo}{\revision{L^2(I\times \Om)}}

\begin{document}

\copyrightyear{2021}
\firstpage{1}


\title[Error Estimates for a State Constrained Stokes Optimal Control Problem]{A priori error estimates for optimal control problems governed by the transient Stokes equations and subject to state constraints pointwise in time}

\author{Dmitriy Leykekhman
\address{\orgdiv{Department of Mathematics}, \orgname{University of Connecticut}, \orgaddress{\street{Storrs}, \postcode{06269}, \state{CT}, \country{USA}}}}
\author{Boris Vexler\ORCID{0000-0001-6211-4263} and Jakob Wagner*\ORCID{0000-0001-8510-9790}
\address{\orgdiv{Chair of Optimal Control}, \orgname{Department of Mathematics, Technical University of Munich}, \orgaddress{\street{Boltzmannstr. 3}, \postcode{85748 Garching b. Munich}, \country{Germany}}}}

\authormark{D. Leykekhman, B. Vexler and J. Wagner}

\corresp[*]{Corresponding author: \href{email:wagnerja@cit.tum.de}{wagnerja@cit.tum.de}}

\received{Date}{0}{Year}
\revised{Date}{0}{Year}
\accepted{Date}{0}{Year}


\abstract{%
    In this paper, we consider a state constrained optimal control problem governed by the transient 
    Stokes equations. The state constraint is given by an $L^2$ functional in space, which is required
    to fulfill a pointwise bound in time. The discretization scheme for the Stokes equations consists of
    inf-sup stable finite elements in space and a discontinuous Galerkin method in time, for which we 
    have recently established best approximation type error estimates. 
    Using these error estimates, for the discrete control problem we \revision{derive} error estimates 
    and as a by-product  we show an improved regularity for 
    the optimal control. We complement our theoretical analysis with numerical results.
}

\maketitle

\section{Introduction}
In this paper we consider the following optimal control problem 
\begin{subequations}
\begin{equation}\label{eq:optimal_problem}
\text{Minimize }\ J(\q,\uu) := \frac{1}{2} \int_0^T\norm{\uu(t) - {\uu}_d(t)}^2_{L^2(\Omega)}\ dt + \frac{\alpha}{2} \int_0^T \norm{\q(t)}^2_{L^2(\Omega)}\ dt
\end{equation}
subject to the state equation
\begin{equation}\label{eq:state_equation}
    \begin{aligned}
     \partial_t \uu - \Delta \uu + \nabla p &= \q  &\quad &\text{in } I \times \Omega,\\
     \nabla\cdot \uu &=0 &\quad  &\text{in } I \times \Omega,\\
      \uu &= \oo &\quad  &\text{on } I \times \partial \Omega,\\
      \uu(0) & = \oo &\quad   &\text{in } \Omega,
    \end{aligned}
\end{equation}
control constraints 
\begin{equation}\label{eq:control}
\q_a \le \q(t,x) \le \q_b \quad \text{for almost all }\; (t,x) \in I\times \Omega
\end{equation}
and state constraints 
\begin{equation}\label{eq:state_constraint}
\int_\Omega \uu(t,x)\cdot {\bf w}(x) \ dx \le \beta \quad \text{for all }\; t \in \bar{I}.
\end{equation}
\end{subequations}
Here we assume that $\Omega\subset \R^d$, $d=2,3$, is a convex polygonal or polyhedral domain and
$I= (0,T]$ is a bounded time interval. 
\revision{In the objective function, $\uu_d \in L^2(I\times \Om)^d$ represents the desired state and $\alpha > 0$ is the regularization parameter.}
The control constraints are given by the constant vectors 
$\q_a,\q_b\in \revC{(\R \cup \{\pm \infty\})^d}$ and satisfy $\q_a < \q_b$.
\revision{In the state constraint, the constant scalar $\beta$ satisfies $\beta > 0$ and
${\bf w}(x)$ is a given function in $L^2(\Om)^d$.}
Note that for ease of presentation, we consider an optimal control problem with 
homogeneous initial data $\uu(0) = \oo$, while 
\revC{all results also extend to the inhomogeneous case $\uu(0) = \uu_0$.}
The main result of this paper states, that the error between the optimal control $\bar{\q}$ for the continuous 
problem and the optimal solution $\bar{\q}_{\sigma}$ of the \revD{discretized} problem satisfies 
\[
\norm{\bar{\q} - \bar{\q}_{\sigma}}_{\LtwoLtwo} \le \frac{C}{\sqrt{\alpha}} \, \lk (k^{\frac{1}{2}}+h),
\]
and is presented in \Cref{thm:main_result}. 
A similar \revC{optimal control problem subject to the heat equation} was considered in
\cite{MeidnerD_RannacherR_VexlerB_2011}, where a 
comparable error estimate was derived. 
\revD{The authors of \cite{ludovici_priori_2015,ludovici_priori_2018} discuss error estimates for 
  parabolic problems with purely time-dependent controls and impose constraints on spatial 
  averages of either function values or gradients of the state at every point in time.
  Error estimates for state constrained parabolic problems, with \revD{state constraints} applied 
  pointwise in time and space, can be found in 
  \cite{christof_new_2021,gong_error_2013,deckelnick_variational_2011}.
}
The optimal control of flow phenomena subject to state constraints is a
very active research topic, and there have been numerous contributions to the field, see, e.g.,
\cite{de_los_reyes_state_constrained_2008,de_los_reyes_semi_smooth_2006, de_los_reyes_regularized_2009} 
for optimal control of the stationary Navier-Stokes equations and 
\cite{fattorini_optimal_1998, liu_optimal_2010,wang_optimal_2002,wang_pontryagin_2002, wang_maximum_2003}
for the transient Navier-Stokes equations, subject to general state constraints of the form 
$\uu \in \mathcal C$.
Note that the above references only contain the analysis of the continuous problems and some numerical results,
but no derivation of error estimates.
\revC{In \cite{reyes_finite_2008} error estimates for an optimal control problem subject to the stationary 
Stokes equations with pointwise state constraints are shown.}
Let us also specifically mention \cite{john_optimal_2009}, where an optimal control problem of the 
stationary Navier-Stokes equations was considered, and a constraint was put onto the drag functional
$\int_{\partial \Omega} (\partial_n \uu - p \mathbf{n}) \mathbf{e}_d \ ds$, for some given direction of
interest 
specified by the unit vector $\mathbf{e}_d$. The setting of our work, constraining a $L^2(\Omega)$
functional pointwise in time, can be seen as a step towards discussing transient problems with 
drag/lift constraints at every point in time.
  The rest of the paper is structured as follows. In \Cref{sec:notation} we introduce the notation  and present some analysis of the transient Stokes problem used in this paper.
  We then proceed to discuss the optimal control problem, including wellposedness and optimality conditions
  in \Cref{sect:cont_opt_control}. 
  \revC{Depending on the regularity of available data, we discuss regularity and structural properties of 
  the optimal solution.}
  We introduce the discretization of the transient Stokes problem in
  \Cref{sec:discretization} and recollect some important error estimates.
  This allows us to discuss the discrete formulations of the optimal control problem, where first 
  in \Cref{sec:var_disc_opt_control}
  we present the analysis and error estimates for a problem with variational discretization,
  \revision{cf., \cite{deckelnick_variational_2011,hinze_variational_2005}}, where only 
  the state equation is discretized, but the control is not.
  Following up this section, we discretize the control by piecewise constant functions in space and time,
  and present the analysis of the fully discrete problem in \Cref{sec:full_disc_opt_control}, which contains
  the main result of this work, \Cref{thm:main_result}.
  We conclude our work by using the derived error estimates to obtain improved regularity 
  for the optimal control in \Cref{sec:improved_regularity} and presenting numerical results in 
  \Cref{sec:num_results}.

\section{Notation and Preliminary results}\label{sec:notation}
We will use the standard notation for the Lebegue and Sobolev spaces
over the spatial domain $\Om$. The pressure space is
\begin{equation*}
  L^2_0(\Om) := \left\{p \in L^2(\Om) \colon \int_\Om p \ dx = 0\right\}.
\end{equation*}
Throughout the paper, vector valued quantities and spaces will be indicated by boldface letters.
We denote for a Banach space $X$ and $1 \le p \le \infty$, by $L^p(I;X)$ the Bochner space of $X$-valued
functions over $I$, whose $X$-norm is $p$-integrable w.r.t time.
If $X$ is reflexive and $1 \le p < \infty$, there holds the following isomorphism,
see \cite[Corollary 1.3.22]{hytonen_bochner_2016}
\begin{equation}\label{eq:bochner_dual_isomorphism}
  L^p(I;X)^* \cong L^q(I;X^*), \qquad \text{where } \frac{1}{p} + \frac{1}{q} = 1.
\end{equation}
Note that the range includes the value $p=1$ and not $p= \infty$.
The dual space of $C(\bar I;L^2(\Om)^d)$ is isomorphic to the 
space of regular $L^2(\Om)^d$-valued Borel measures, and we denote it by 
$\mathcal M(\bar I;L^2(\Om)^d) \cong (C(\bar I;\revD{L^2(\Om)^d}))^*$.
Similarly for scalar regular Borel measures of $\bar I$, we use the notation 
$\mathcal M(\bar I) \cong (C(\bar I))^*$.
To denote the vector valued spaces of divergence free functions, having various levels of regularity,
we use the following notation
\begin{equation*}
  \Vz = \{\vv \in L^2(\Omega)^d: \ \nabla \cdot \vv = 0 \land \vv \cdot \mathbf{n} = 0 
  \text{ on } \partial \Om\}, \quad
    \V:=\{\vv\in H^1_0(\Omega)^d: \ \nabla\cdot\vv=0 \}, \quad
    \Vt := H^2(\Omega)^d \cap \V,
\end{equation*}
\revB{where by $\vv\cdot \mathbf{n}$ we denote the generalized normal trace.}
We denote by $\V^*$ the topological dual space of $\V$, and define
\revC{%
\begin{align}\label{eq:W_imbedding}
  \W &:= L^2(I;\V)\cap H^1(I;\V^*)\hookrightarrow C(\bar{I};\revC{\Vz}), \notag\\
  \X &:= \{\vv \in L^2(I;\V): \partial_t \vv \in L^2(I;\V^*) + L^1(I;\Vz) \text{ and } \vv(0)=\oo\}, \notag \\
  \Y &:= L^2(I;\V^*) + L^1(I;\Vz)  \quad \Rightarrow \quad \Y^* \cong L^2(I;\V) \cap L^\infty(I;\Vz) \quad
  \text{due to \eqref{eq:bochner_dual_isomorphism}.}\notag
\end{align}
}
We denote by $\Oprod{\cdot,\cdot}$ and $\IOprod{\cdot,\cdot}$ the inner products of $L^2(\Omega)^d$ and 
$L^2(I;L^2(\Omega)^d)$ respectively, and
by $\IOpair{\cdot,\cdot}$ the duality pairing between $L^2(I;\V)$ and $L^2(I;\V^*)$.
  We introduce the Stokes operator $A$, defined by
\begin{equation*}
  A \colon D(A) \subset \Vz \to \Vz, \ \Oprod{A \uu,\vv} = \Oprod{\nabla \uu,\nabla \vv}.
\end{equation*}
The $H^2$ regularity results of \cite{dauge_stationary_1989,kellogg_regularity_1976} show that 
$D(A) = \Vt$. As $A$ is a positive, selfadjoint operator, fractional powers $A^s$ are well defined.
Of special importance is $A^{\frac{1}{2}}$ which is an isometric isomorphism
\begin{equation*}
  A^{\frac{1}{2}} \colon D(A^{\frac{1}{2}}) = \V \to \Vz,
\end{equation*}
as it holds 
$\|A^{\frac{1}{2}} \uu\|_{L^2(\Om)}^2 = \Oprod{A \uu,\uu} = \Oprod{\nabla \uu,\nabla \uu} 
= \|\nabla \uu\|_{L^2(\Om)}^2$.
For the proof of $D(A^{\frac{1}{2}}) = \V$, see \cite[Ch. III, Lemma 2.2.1]{sohr_navier-stokes_2001}.
By its definition, we can extend $A$ to an operator (denoted by the same symbol) $A \colon \V \to \V^*$
yielding another isometric isomporphism between those spaces.
Lastly, as 
\begin{equation*}
  \|A^{\frac{1}{2}}\uu\|_{\V^*} 
  = \sup_{\vv \in \V} \dfrac{\Opair{A^{\frac{1}{2}} \uu,\vv}}{\|\vv\|_{\V}}
  = \sup_{\ww \in \Vz} \dfrac{\Opair{A^{\frac{1}{2}} \uu, A^{-\frac{1}{2}}\ww}}{\|A^{-\frac{1}{2}} \ww\|_{\V}}
  = \sup_{\ww \in \Vz} \dfrac{\Oprod{\uu, \ww}}{\|A^{-\frac{1}{2}} \ww\|_{\V}}
  = \sup_{\ww \in \Vz} \dfrac{\Oprod{\uu, \ww}}{\|\ww\|_{\Vz}}
  = \|\uu\|_{\Vz},
\end{equation*}
there holds the isometric isomorphism
\begin{equation}\label{eq:isom_isom_VzVstar}
  A^{\frac{1}{2}} \colon \Vz \to \V^*.
\end{equation}
Using the above defined function spaces, the weak formulation of the state equation \eqref{eq:state_equation} 
for a given \revC{$\q\in L^1(I;L^2(\Om)^d) + L^2(I;\V^*)$} reads as follows: Find $\uu\in \revC{\X}$ such that 
\revC{%
\begin{equation}\label{eq:weak_state}
\begin{aligned}
    \Opair{\pa_t \uu,\vv}+\Oprod{\na \uu,\na \vv} 
    &= \Oprod{\q,\vv}\quad \revC{\text{for all } \vv \in \V, \ \text{a.e. in } I}. 
\end{aligned}
\end{equation}
}
For the above weak formulation, there holds the following result.
\begin{theorem}\label{thm:solvability_state_equation}
    For $\q \in L^1(I;\revC{L^2(\Om)^d}) + L^2(I;\V^*)$, there exists a unique solution
    $\uu \in L^2(I;\V) \cap C(\bar I;\Vz)$ solving \eqref{eq:weak_state} and the following estimate holds
    \begin{equation*}
        \|\uu\|_{L^2(I;\V)} + \|\uu\|_{L^\infty(I;L^2(\Om))}
        \le C \|\q\|_{L^1(I;\Vz) + L^2(I;\V^*)}.
    \end{equation*}
    If $\q \in L^2(I;\revC{L^2(\Om)^d})$, then
    $\uu \in L^2(I;\revC{\Vt}) \cap H^1(I;\revC{\Vz}) \hookrightarrow C(\bar I;\V)$, and there holds
    \begin{equation*}
        \|\uu\|_{L^2(I;H^2(\Om))} + \|\uu\|_{L^\infty(I;\V)} + \|\partial_t \uu\|_{\LtwoLtwo}
        \le C \|\q\|_{\LtwoLtwo} .
    \end{equation*}
\end{theorem}
\begin{proof}
    The first part of this theorem is proven in \cite[Chapter III, Theorem 1.1]{1977Temam}
    and the remark on page 179 therein.
    The $H^2$ regularity part can be shown as in \cite[Chapter III, Proposition 1.2]{1977Temam}
    using the $H^2$ regularity result for the stationary Stokes equations in convex
    polygonal/polyhedral domains, proven in \cite[Theorem 5.5, Theorem 6.3]{dauge_stationary_1989}
    see also \cite[Theorem 2]{kellogg_regularity_1976}.
\end{proof}

\revC{%
  It is a classical result, that the Stokes operator in the Hilbert space setting exhibits maximal parabolic
  regularity, i.e.
  \begin{equation}\label{eq:max_par_reg_L2}
    \q \in L^p(I;\Vz), 1<p<\infty 
    \quad \Rightarrow \quad \uu \text{ solving \eqref{eq:weak_state} satisfies } \partial_t \uu, A\uu \in L^p(I;\Vz),
  \end{equation}
  see \cite[Proposition 2.6]{Behringer_Leykekhman_Vexler_2023},
  holding on fairly general domains, e.g. Lipschitz domains. If $\Omega$ is convex, the $H^2$ regularity results
  further imply that $\uu \in L^p(I;\Vt)$.
  Most often maximal parabolic regularity is treated in $L^2$ or $L^q$ setting in space,
  but it can also be extended 
  to settings of weaker spacial regularity. The property \eqref{eq:isom_isom_VzVstar} combined with 
  \cite[Lemma 11.4]{auscher_square_2015} yields that maximal parabolic regularity also holds in $\V^*$, i.e.
  \begin{equation}\label{eq:max_par_reg_dual}
    \q \in L^p(I;\V^*), 1 < p < \infty
    \quad \Rightarrow \quad \uu \text{ solving \eqref{eq:weak_state} satisfies } 
    \partial_t \uu \in L^p(I;\V^*), \uu \in L^p(I;\V).
  \end{equation}
  }
\revC{If the right hand side $\q$ is regular enough, there exists an associated pressure to the
weak solution of \eqref{eq:weak_state}, formulated in divergence free spaces.
The regularity of the pressure depends on the regularity of the right hand side 
and the velocity component of the solution. There holds the following result, 
see \cite[Theorem 2.10, Corollary 2.11]{Behringer_Leykekhman_Vexler_2023}.}
\begin{proposition}\label{prop:associated_pressure}
  Let $\q \in L^s(I;L^2(\Om)^d)$ for some $1< s < \infty$ and let
  $\uu \in L^2(I;\V) \cap \revC{C(\bar I;\Vz)}$ be the weak solution to
  \eqref{eq:weak_state}. Then there exists \revC{a unique} $p \in L^s(I;L^2_0(\Om))$ such that 
  \begin{equation}\label{eq:associated_pressure}
    \partial_t \uu - \Delta \uu + \nabla p = \q,
  \end{equation}
  which is to be understood as an identity in $L^s(I;H^{-1}(\Om)^d)$.
  \revC{On convex domains $\Omega$, the pressure satisfies $p \in L^s(I;H^1(\Om))$.}
\end{proposition}
\section{Continuous Optimal Control Problem}\label{sect:cont_opt_control}
\revC{We now} introduce the control to state mapping.
\begin{theorem}\label{thm:regularity_sol_operator}
Let $S\colon {\q}\mapsto {\uu}$ denote the solution operator for the state  equation 
\eqref{eq:weak_state}. Then $S$ is a bounded linear operator between the following spaces;
\begin{itemize}
\item $S\colon L^2(I;\V^*)\to \W \hookrightarrow C(\bar{I};L^2(\Omega)^d)$
\item $S\colon L^1(I;\Vz)\to C(\bar{I};L^2(\Omega)^d)$
\item $S\colon L^\infty(I;L^2(\Om)^d)\to W^{1,s}(I;L^2(\Om)^d)\cap L^{s}(I;H^2(\Om)^d)$, \quad $1 \le s<\infty$.
\end{itemize}
\end{theorem}
\begin{proof}
  The first two claims are the direct consequences of \Cref{thm:solvability_state_equation}. Note, that 
  the regularity $\partial_t \uu \in L^2(I;\V^*)$ can only be obtained by bootstrapping if the right hand side
  is in $L^2(I;\V^*)$. The third claim is obtained by using the maximal parabolic regularity
  and $H^2$ regularity of the stationary Stokes problem on convex polygonal/polyhedral domains, see 
  \cite[Theorem 5.5, Theorem 6.3]{dauge_stationary_1989}.
\end{proof}

To abbreviate the notation, we will frequently use $\uu(\q) := S(\q)$.
Since $S$ is a  bounded linear operator between the spaces introduced in the previous theorem,
it is Fr\'echet differentiable, and its \revC{directional derivative in direction $\dq$} satisfies
\begin{equation*}
  S'(\q)(\dq) = S(\dq),
\end{equation*}
i.e. is independent of $\q$.
\begin{remark}
    The operator $S$ is linear and coincides with its Frechét derivative,
    due to our choice to work with homogeneous inital data
    in the state equation \eqref{eq:state_equation}. The results presented in this work however also
    hold true in the inhomogeneous initial data case. In this case, the 
    assumption $\beta>0$ \revision{, on the parameter $\beta$ of the state constraint,}
    then has to be generalized to $\beta> \Oprod{\uu_0,\ww}$.
\end{remark}
%
\revC{Let us next give a characterization of the adjoint operator $S^*$.}
\Cref{thm:solvability_state_equation}, 
\revC{together with \cite[Proposition V.1.3]{boyer_mathematical_2013},} 
yields that for any $\q \in \revC{\Y}$, there 
exists a unique solution $\uu \in \X$ such that 
\begin{equation}\label{eq:weak_state_timedept_testfunctions}
  a(\uu,\vv) := \IOpair{\partial_t \uu,\vv} + \IOprod{\nabla \uu,\nabla \vv} = \IOpair{\q,\vv} \qquad
  \forall \vv \in \revC{\Y^*},
\end{equation}
\revC{and the two formulations \eqref{eq:weak_state} and \eqref{eq:weak_state_timedept_testfunctions} are 
equivalent.}
The solution operator $S$ is the inverse of the operator $T\colon \X \to \Y$, 
$\langle T \uu,\vv\rangle_{\Y\times \Y^*} = a(\uu,\vv)$.
As $T$ is invertible, so is $T^*$ and its inverse is precisely $S^*$. 
With this construction, the adjoint operator $S^*: \mathbf{g} \mapsto \zz$,
corresponds to the following weak formulation:
given $\mathbf{g} \in \X^*$, find $\zz \in \revC{\Y^*}$ satisfying
\begin{equation}\label{eq:weak_formulation_adjoint}
  \IOpair{\partial_t \vv,\zz} + \IOprod{\nabla \vv,\nabla \zz} 
  = \IOprod{\mathbf{g},\vv} \quad \forall \vv \in \X.
\end{equation}

We now discuss the regularity of the adjoint operator.

\begin{corollary}\label{cor:adjoint_sol_operator_regularity}
  Let $S^*$ denote the adjoint operator to the solution operator introduced in \Cref{thm:regularity_sol_operator}. Then it satisfies
  \begin{equation*}
    S^*\colon \mathcal M(\bar I;L^2(\Omega)^d) \to L^2(I;\V)\cap  L^\infty(I;\Vz).
  \end{equation*}
\end{corollary}
\begin{proof}
  This is a direct consequence of the definition of the adjoint operator, \Cref{thm:regularity_sol_operator},
  and the isomorphism $(L^1(I;\Vz))^* \cong L^\infty(I;\Vz)$, due to \eqref{eq:bochner_dual_isomorphism}.
\end{proof}

Due to the linearity of the adjoint operator, it again holds
\begin{equation*}
  S^*(\dq) = (S'(\q))^*(\dq).
\end{equation*}
For convenience, for the weight ${\bf w}\in L^2(\Om)^d$,
we define the functional $G_{\bf w}\colon L^2(\Om)^d\to \mathbb{R}$ by
\begin{equation}\label{eq: definition of G}
G_{\bf w}(\vv):=(\vv,{\bf w})_{\Omega}.
\end{equation}
For time-dependent functions $\vv \colon I\to L^2(\Om)^d$, the application of $G_{{\bf w}}$ is defined by
$$
G_{\bf w}(\vv)(t):=G_{\bf w}(\vv(t)).
$$
Using the functional $G_{\bf w}$, the state constraint \eqref{eq:state_constraint} can be compactly rewritten as
\begin{equation}\label{eq: state constrain using G}
G_{\bf w}(\uu)\le \beta\quad \text{in}\ \bar{I}.
\end{equation}
\begin{remark}\label{remark: continuity of G}
  Due to the continuous embedding $\W\hookrightarrow C(\bar{I};L^2(\Omega)^d)$, we have $G_{\bf w}(\vv)(\cdot)\in C(\bar{I})$ for any $\vv\in \W$. Thus we can understand $G_{\bf w}$ as a linear, continuous operator
  from $\W$ to $C(\bar I)$.
\end{remark}

To write the optimal control problem \eqref{eq:optimal_problem}-\eqref{eq:state_constraint} into 
reduced form, we define $\mathcal{G}:=G_{\bf w}\circ S$ and the closed convex cone
$\mathcal{K}\subset C(\bar{I})$ by
$$
\mathcal{K}:=\{v\in C(\bar{I}): \ v\le \beta\ \text{in}\ \bar{I}\}.
$$
Using the above definitions, the reduced form reads
\begin{equation}\label{optimal problem reduced}
\text{Minimize }\ j(\q):=J(\q,S(\q)) \quad \text{for} \quad \q\in \Q_{ad}\quad \text{subject to}\quad 
\mathcal{G}(\q)\in \mathcal{K},
\end{equation}
where the admissible set is given by
$$
\Q_{ad}=\{ \q\in L^2(I; L^2(\Omega)^d):
\quad \q_a \le \q(t,x) \le \q_b \quad \text{a.e. in }   I\times \Omega\}.
$$
\revB{%
We define the projection operator onto the feasible set by
\begin{equation*}
  P_{[\q_a,\q_b]}(\q) := \min\{\q_b,\max\{\q,\q_a\}\},
\end{equation*}
which we understand componentwise and pointwise for every $(t,x) \in I \times \Om$.
Throughout the paper, we will work under the following assumption.
}

\begin{assumption}[Slater condition]\label{Slater}
  There exists $\tilde{\q}\in \Q_{ad}$ such that $G_{\bf w}(\uu(\tilde{\q}))<\beta$ for 
  \revC{all $t \in \bar I$}, where
  $\uu(\tilde{\q})$ is the solution of the weak transient Stokes problem \eqref{eq:weak_state} for this
  particular control $\tilde{\q}$.
\end{assumption}

\begin{remark}
  As the homogeneous initial data in our setting necessitate the choice $\beta > 0$, if the control 
  constraints $\q_a, \q_b$ admit the control $\tilde \q \equiv \oo$, the existence of such a Slater point is 
  immediately given.
  In that case, the unique solution to the state equation with right hand side $\tilde \q \equiv \oo$ is 
  $\uu(\tilde \q) \equiv \oo$ which trivially satisfies $\Gw(\uu(\tilde \q)) = 0 < \beta$.
\end{remark}
\begin{theorem}\label{thm:existence_optimal_control}
  Under \Cref{Slater}, there exists a unique optimal control $\bar \q \in L^2(I;L^2(\Om)^d)$
  with unique associated state $\bar \uu$, solving the optimal control problem
  \eqref{eq:optimal_problem}-\eqref{eq:state_constraint}.
\end{theorem}
\begin{proof}
  The \Cref{Slater} yields the existence of a feasible $\tilde \q$, such that the associated state 
  $\tilde \uu = \uu(\tilde \q)$ satisfies the state constraint. Let us define
  \begin{equation*}
    \mathcal J := \inf \{J(\q,\uu)\colon \q 
    \text{ and $\uu$ satisfy \eqref{eq:state_equation}-\eqref{eq:state_constraint}}\}
    \revC{\ge 0},
  \end{equation*}
  and let $\{\q_n\}$ and $\{\uu_n := \uu(\q_n)\}$ denote sequences of feasible controls with associated 
  states, such that $J(\q_n,\uu_n) \to \mathcal J$ as $n\to \infty$.
  As it holds $J(\q_n,\uu_n) \le \revC{\mathcal{J}+1}$ for large enough $n$, there holds a bound
  $\|\q_n\|_{\LtwoLtwo} \le C$ for all $n$.
  From \Cref{thm:regularity_sol_operator}, we obtain 
  $\|\uu_n\|_{L^2(I;H^2(\Om))} + \|\uu_n\|_{H^1(I;L^2(\Om))} \le C$. 
  We can thus take a subsequence, denoted by the same index, such that 
  \begin{equation*}
    \q_n \rightharpoonup \hat \q \text{ in } L^2(I;L^2 (\Om)^d), \quad
    \uu_n \rightharpoonup \hat \uu \text{ in } L^2(I;H^2(\Om)^d), \quad
    \partial_t \uu_n \rightharpoonup \partial_t \hat \uu \text{ in } L^2(I;L^2 (\Om)^d).
  \end{equation*}
  These allow us to pass to the limit in the weak form of the state equation, showing that 
  $\hat \uu = \uu(\hat \q)$. Furthermore, as 
  $L^2(I;H^2(\Om)^d)\cap H^1(I;L^2(\Om)^d) \hookrightarrow C(\revD{\bar I};H^s(\Om)^d)$ compactly,
  for $s<1$, see \cite[Corollary 8]{simon_compact_1986}, by taking another subsequence, we obtain 
  $\uu_n \to \hat \uu$ in $C(\revD{\bar I};L^2(\Om)^d)$. 
  As $\Gw(\uu_n) \le \beta$ for \revC{all} $t \in I$, this 
  shows $\Gw(\hat \uu) \le \beta$ for \revC{all} $t \in I$. Lastly, it holds due to the lower semicontinuity of 
  the norms
  \begin{equation*}
    J(\hat \q,\hat \uu) \le \liminf_{n\to \infty} J(\q_n,\uu_n) = \mathcal J,
  \end{equation*}
  which shows that $\bar \q = \hat \q$ is indeed a minimizer with associated state $\bar \uu = \hat \uu$.
  Using uniform convexity of the squared $L^2$ norms and linearity of the state equation gives the uniqueness.
\end{proof}

\begin{theorem}[First order optimality system]\label{thm:continuous first order optimality}
  \revC{Let \Cref{Slater} be fulfilled. Then a}
  control $\bar{\q}\in \Q_{ad}$ with associated state $\bar{\uu}=\uu(\bar{\q})$ is \revC{the} optimal solution
  to the problem \eqref{eq:optimal_problem}-\eqref{eq:state_constraint} if and only if  and there exists 
  an adjoint state $\bar{\zz}\in L^2(I;\V)\cap \revB{L^\infty(I;\Vz)}$ and a Lagrange multiplier 
  $\bar{\mu}\in (C(\bar I))^*$ that satisfy:\\
\begin{subequations}
{State equation}
\begin{equation}\label{eq: optimal state}
    \begin{aligned}
     \partial_t \bar{\uu} - \Delta \bar{\uu} + \nabla \bar{p} &= \bar{\q}  &\quad &\text{in } I \times \Omega,\\
     \nabla\cdot \bar{\uu} &=0 &\quad  &\text{in } I \times \Omega,\\
     \bar{\uu} &= \oo &\quad  &\text{on } I \times \partial \Omega,\\
     \bar{\uu}(0) & = \oo&\quad   &\text{in } \Omega;
    \end{aligned}
\end{equation}
{\revB{State constraint and complementarity conditions}}
\begin{equation}\label{eq: Complementary conditions}
  \Gw(\bar{\uu})\le \beta , \quad
    \bar{\mu}\geq 0\quad\text{and}\quad \langle \bar{\mu},\beta-\Gw (\bar \uu) \rangle =0;
\end{equation}
{Adjoint equation}
\begin{equation}\label{eq: adjoint state}
    \begin{aligned}
     -\partial_t \bar{\zz} - \Delta \bar{\zz} + \nabla \bar{r} &= \bar{\uu}-\uu_d+{\bar{\mu}}\ww  &\quad &\text{in } I \times \Omega,\\
     \nabla\cdot \bar{\zz} &=0 &\quad  &\text{in } I \times \Omega,\\
     \bar{\zz} &= \oo &\quad  &\text{on } I \times \partial \Omega,\\
     \bar{\zz}(T) & = \revB{\bar \mu(\{T\}) \ww} &\quad   &\text{in } \Omega;
    \end{aligned}
\end{equation}
{Variational inequality }
\begin{equation}\label{eq: Variational inequality}
\left(\alpha\bar{\q}+\bar{\zz}, \delta \q-\bar{\q}\right)\geq 0 \quad  \forall \delta \q\in \Q_{ad} \quad \Leftrightarrow \quad \bar{\q} =P_{[\q_a,\q_b]}\left(-\frac{1}{\alpha}\bar{\zz}\right).
\end{equation}
\end{subequations}
 The state \eqref{eq: optimal state} and adjoint \eqref{eq: adjoint state} equations should be understood
 in the weak sense \revB{\eqref{eq:weak_state} and \eqref{eq:weak_formulation_adjoint} respectively.}
\end{theorem}
\begin{proof}
  We derive the optimality conditions from the reduced problem \eqref{optimal problem reduced}.
  By the generalized KKT condition (cf. \cite[Theorem 5.2]{casas_boundary_1993}) and the Slater condition
  from Assumption \ref{Slater}, i.e.,
  the existence of $\tilde{\q}\in \Q_{ad}$ such that $\mathcal{G}(\tilde{\q})\in \operatorname{int}\mathcal{K}$,
  the optimality of $\bar{\q}$ is equivalent to the existence of a Lagrange multiplier 
  $\bar{\mu}\in C(\bar{I})^*$ and the adjoint state 
\begin{equation}\label{eq: definition of z from S*}
\bar{\zz}= S'(\bar{\q})^*(S(\bar{\q})-\uu_d)+\mathcal{G}'(\bar{\q})^*\bar{\mu}
\end{equation}
 satisfying 
 $$
 \left(\alpha\bar{\q}+\bar{\zz}, \delta \q-\bar{\q}\right)\geq 0 \quad  \forall \delta \q\in \Q_{ad} \quad \text{and}\quad \langle v
 -\mathcal{G}(\bar{\q}),\bar{\mu}\rangle\le 0\quad \forall v\in \mathcal{K}. 
 $$
 By definition and linearity of the involved operators,
 we can write $\mathcal{G}'(\bar \q)^* = S^* \circ G_{\mathbf w}^*$, where 
 $G_{\mathbf w}^*\colon C(\bar I)^* \to C(\bar I; L^2(\Om)^d)^*$, $G_{\mathbf w}^*(\mu) = \mu \ww$.
 Thus it holds $\bar \zz = S^*(\bar \uu - \uu_d + \bar \mu \ww)$, which gives the proposed 
 regularity of $\bar \zz$ as a consequence of \Cref{cor:adjoint_sol_operator_regularity}.
To complete the proof, we point to the following equivalence
$$
(S(\bar \q) \in \mathcal K) \land (\langle v-\mathcal{G}(\bar{\q}),\bar{\mu}\rangle\le 0\quad \forall v\in \mathcal{K})
\quad \Leftrightarrow \quad
(G_{\ww}(\bar \uu) \le \beta) \land
(\bar{\mu}\geq 0) \land (\langle \beta-{G}(\bar{\uu}),\bar{\mu}\rangle= 0),
$$
which is exactly \eqref{eq: Complementary conditions}.
\end{proof}


From the first order optimality system \eqref{eq: optimal state}-\eqref{eq: Complementary conditions} we can derive the following regularity results.  

\begin{theorem}\label{thm: regularity optimal}
Let $\bar{\q}$ denote the optimal  control for the problem \eqref{eq:optimal_problem}. Then the following regularity holds:
$$
\bar{\q}\in L^2(I;H^1(\Om)^d)\cap \revB{L^\infty(I;L^2(\Om)^d)}.
$$
\revB{%
If additionally, $\q_a,\q_b \in \R^d$, i.e. are finite, it holds 
$\bar{\q}\in \revC{L^\infty(I \times \Om)^d}$.
}
\end{theorem}
\begin{proof}
Since ${\bf w}\in L^2(\Om)^d$ and $\bar{\mu}\in (C(\bar{I}))^*$, we have that 
$\bar{\mu}\ww\in ( C(\bar{I};L^2(\Omega)^d))^*$. 
By \Cref{cor:adjoint_sol_operator_regularity}, it holds 
$\bar \zz = S^*(\bar \uu - \uu_d + \bar \mu \ww) \in L^2(I;\V)\cap L^\infty(I;L^2(\Om)^d)$.
Using that 
$$
\bar{\q}= P_{[\q_a,\q_b]} \left(-\frac{1}{\alpha}\bar{\zz}\right),
$$
and $\q_a, \q_b$ are constant, we have the theorem.
\end{proof}
  The above result shows the regularity available for the optimal control $\bar \q$ without any additional
  assumptions. If one assumes higher regularity of the data, especially of the weight $\ww$ of the 
  state constraint, we can show improved regularity, and in some cases even a structural result, that 
  the Lagrange multiplier $\bar \mu$ contains no Dirac contributions. We first treat a general adjoint 
  problem with measure valued right hand side in \Cref{thm:adjoint_measure_regularity}, before 
  showing regularity of the optimal variables in \Cref{thm:optimal_regularity_higher_smoothness} and the 
  corollaries thereafter.
  \begin{theorem}\label{thm:adjoint_measure_regularity}
    Let $\mu \in C(\bar I)^*$ and \revD{$\ww \in L^2(\Om)^d$} be given,
    and let $\zz \in L^2(I;\V) \cap L^\infty(I;\Vz)$  
    be the weak solution to
    \begin{equation}\label{eq:weak_formulation_adjoint_general_measure_rhs}
      \IOpair{\partial_t \vv, \zz} + \IOprod{\nabla \vv,\nabla \zz} 
      = \IOpair{\vv,\ww \mu} \quad \text{for all } \vv \in \X.
    \end{equation}
    Then, if additionally $\ww \in \V$, it holds $\zz \in L^s(I;\V) \cap \BV(I;\V^*)$ 
    for all $1 \le s < \infty$.
    Moreover, if $\ww \in \Vt$, it holds $\zz \in L^s(I;\Vt) \cap \BV(I;\Vz)$ for all $1 \le s < \infty$.
  \end{theorem}
  \begin{proof}
    Throughout this proof, let $\ww$ satisfy at least $\ww \in \V$.
  Using \cite[Theorem 4.31]{appell_bounded_2013}, there exists a normalized function 
  of bounded variation $\tilde \mu \in \NBV(\bar I)$, such that 
  the application of $\mu \in C(\bar I)^*$ to any $\xi \in C(\bar I)$ 
  can be expressed as a Riemann-Stieltjes integral
  \begin{equation*}
    \revision{\Ipair{\xi,\mu}} = \int_I \xi(t) \ \d \tilde \mu(t).
  \end{equation*}
  Slightly modifying \cite[Definition 1.2]{appell_bounded_2013}, we can normalize $\tilde \mu$
  such that it is of bounded variation, right continuous and satisfies $\tilde \mu(T) = 0$.
  If $\xi \in C^1(\bar I)$, due to \cite[Proposition 4.24, Theorem 4.17]{appell_bounded_2013}, there holds
  \begin{equation*}
    \revision{\Ipair{\xi, \mu}} 
    = - \int_I \tilde \mu(t)\ \d \xi(t) + \tilde \mu(T) \xi(T) - \tilde \mu(0) \xi(0)
    = - \int_I \tilde \mu(t) \partial_t \xi(t) \ \d t -\tilde \mu(0) \xi(0) 
  \end{equation*}
  where the last integral can be understood equivalently in the Riemann or Lebesgue sense.
  Working with the Lebesque integral allows us to pass to the limit 
  $C^1(\bar I) \ni \xi_n \to \xi \in W^{1,1}(I) \hookrightarrow C(\bar I)$, showing that 
  \begin{equation*}
    \Ipair{\xi, \mu} 
    = - \Iprod{\tilde \mu, \partial_t \xi} - \tilde \mu(0) \xi(0) 
    \qquad \forall \xi \in W^{1,1}(I).
  \end{equation*}
  As due to the definition of $\X$, it holds 
  $\partial_t \vv \in L^1(I;\Vz) + L^2(I;\V^*)$
  and $\vv(0) = \oo$ for all $\vv \in \X$, and since $\ww \in \V$, we have $\Opair{\vv,\ww} \in W^{1,1}(I)$.
  Thus we obtain
  \begin{equation}\label{eq:stateconstraint_identity_new}
    \IOpair{\vv,\ww \mu}
    = - \IOpair{\tilde \mu(t) \ww, \partial_t \vv}. 
  \end{equation}
Let us define $\hat \zz =  \zz + \tilde \mu \ww$. Then adding
\eqref{eq:weak_formulation_adjoint_general_measure_rhs}, \eqref{eq:stateconstraint_identity_new}
and the identity 
  $\IOprod{\nabla \vv, \tilde \mu\nabla \ww} = \IOpair{\tilde \mu A \ww,\vv}$ yields
  \begin{equation*}
    \IOpair{\partial_t \vv,\hat \zz} + \IOprod{\nabla \vv,\nabla \hat \zz}
    = \IOpair{\tilde \mu A \ww,\vv}
    \quad \forall \vv \in \X.
  \end{equation*}
  As $\tilde \mu \in L^\infty(I)$ and $\tilde \mu(T) < \infty$, $\hat \zz$ satisfies a backwards in time 
  Stokes equation with right hand side $\tilde \mu A \ww$.
  Depending on the regularity of $\ww$, we obtain the following:\\
  \textit{Case 1:} $\ww \in \V$. Here $\tilde \mu \ww \in \BV(I;\V) \hookrightarrow L^\infty(I;\V)$ 
      and thus $\tilde \mu A \ww \in L^\infty(I;\V^*)$. 
      Using the maximal parabolic regularity in $\V^*$ \eqref{eq:max_par_reg_dual}, this yields
      \begin{equation*}
        \hat \zz \in W^{1,s}(I;\V^*) \cap L^s(I;\V) \quad \text{for any } 1 \le s < \infty.
      \end{equation*}
      Considering the special cases $s=1$ and $s=2$, we obtain
      \begin{equation*}
        \hat \zz \in W^{1,1}(I;\V^*) \hookrightarrow \BV(I;\V^*) \quad \text{and} \quad
        \hat \zz \in H^{1}(I;\V^*) \cap L^2(I;\V) \hookrightarrow C(\bar I;\Vz).
      \end{equation*}
      As a consequence $ \zz = \hat \zz - \tilde \mu \ww \in L^s(I;\V) \cap \BV(I;\V^*)$ for any 
      $1 \le s < \infty$.
      \\
      \textit{Case 2:} $\ww \in \Vt$. Here $\tilde \mu \ww \in \BV(I;\Vt) \hookrightarrow L^\infty(I;\Vt)$
      and thus $\tilde \mu A \ww \in L^\infty(I;\Vz)$. 
      Using the maximal parabolic regularity in $\Vz$ \eqref{eq:max_par_reg_L2}, this yields
      \begin{equation*}
        \hat \zz \in W^{1,s}(I;\Vz) \cap L^s(I;\Vt) \quad \text{for any } 1 \le s < \infty.
      \end{equation*}
      Considering the special cases $s=1$ and $s=2$, we obtain
      \begin{equation*}
        \hat \zz \in W^{1,1}(I;\Vz) \hookrightarrow \BV(I;\Vz) \quad \text{and} \quad
        \hat \zz \in H^{1}(I;\Vz) \cap L^2(I;\Vt) \hookrightarrow C(\bar I;\V).
      \end{equation*}
      As a consequence $ \zz = \hat \zz - \tilde \mu \ww \in L^s(I;\Vt) \cap \BV(I;\Vz)$ for any 
      $1 \le s < \infty$.
  \end{proof}
  Since in the above proof we only have shown continuity of $\hat \zz$, but $\tilde \mu$ in general can 
  have jumps as it is a $\BV$ function, without further information, we cannot deduce continuity of 
  the whole solution $\zz$. For the adjoint equation \eqref{eq: adjoint state} this translates to the question
  of continuity in time of $\bar \zz$ and by \eqref{eq: Variational inequality} also of $\bar \q$.
  If in addition to the arguments of the previous theorem, we exploit the information from the 
  optimality system, it turns out, that we can at least show continuity in time of $\bar \q$.
  As the adjoint state in the optimality system 
  has another term $\uu - \uu_d$ on the right hand side, depending on the available regularity of $\uu_d$ 
  in time, we might lose some of the time regularity that we just derived.
  \begin{theorem}\label{thm:optimal_regularity_higher_smoothness}
    Let $(\bar \q,\bar \uu,\bar \zz,\bar \mu)$ satisfy the first order necessary optimality conditions
    \eqref{eq: optimal state}-\eqref{eq: Variational inequality}.
    Let additionally $\uu_d \in L^s(I;L^2(\Om)^d)$ for some $s \in [2,\infty)$ and $\ww \in \V$.
    Then $\bar \zz \in L^s(I;\V) \cap \BV(I;\V^*)$ and
    $\bar \q \in L^s(I;H^1(\Om)^d) \cap C(\bar I;L^2(\Om)^d)$.
    If further $\ww \in \Vt$, then $\bar \zz \in L^s(I;\Vt) \cap \BV(I;\Vz)$ and 
    $\bar \q \in L^s(I;W^{1,\infty}(\Om)^d) \cap C(\bar I;H^1(\Om)^d)$.
  \end{theorem}
  \begin{proof}
    As in the proof of the previous theorem, we introduce the $\NBV$ function $\tilde \mu$ satisfying
    $\Ipair{\xi,\bar \mu} = - \Iprod{\tilde \mu,\partial_t \xi} - \tilde \mu(0)\xi(0)$ 
    for all $\xi \in W^{1,1}(I)$,
    and introduce $\hat \zz = \bar \zz + \tilde \mu \ww$. This now satisfies
  \begin{equation*}
    \IOpair{\partial_t \vv,\hat \zz} + \IOprod{\nabla \vv,\nabla \hat \zz}
    = \IOprod{\uu - \uu_d,\vv} + \IOpair{\tilde \mu A \ww,\vv}
    \quad \forall \vv \in \X.
  \end{equation*}
  The regularity $\bar \uu \in L^\infty(I;\Vz)$, the assumption $\uu_d \in L^s(I;L^2(\Om)^d)$ and the previous 
  arguments again yield that
  \begin{align*}
    \ww \in \V &\Rightarrow \hat \zz \in W^{1,s}(I;\V^*) \cap L^s(I;\V) 
    \hookrightarrow C(\bar I;\Vz) \cap \BV( I;\V^*) \ \text{ and } \
    \bar \zz \in L^s(I;\V),\\
    \ww \in \Vt &\Rightarrow \hat \zz \in W^{1,s}(I;\Vz) \cap L^s(I;\Vt) 
    \hookrightarrow C(\bar I;\V) \cap \BV( I;\Vz)
    \ \text{ and } \ \bar \zz \in L^s(I;\Vt).
  \end{align*}
  As in \Cref{thm:adjoint_measure_regularity} this shows the $\BV$ regularity in time of $\bar \zz$.
  To show the regularity of $\bar \q$, we shall make use of the optimality conditions.
  The optimal control satisfies
  \begin{equation}\label{eq:control_representation}
    \bar \q = P_{[\q_a,\q_b]}(\zz_\alpha) \quad \text{where} \quad 
    \zz_\alpha := - \frac{1}{\alpha}\bar \zz =- \frac{1}{\alpha} \hat \zz + \frac{1}{\alpha}\tilde \mu \ww
    \in C(\revD{\bar I};\Vz) + \BV(I;\V).
  \end{equation}
  The available regularity of $\bar \zz$ immediately gives the claimed $L^s$ regularity in time of $\bar \q$.
  Note that due to the application of $\Pad$, $\bar \q$ is in general not divergence free anymore and 
  exhibits at most $W^{1,\infty}(\Om)$ regularity in space, even if $\bar \zz$ is smoother.
  We now turn towards showing continuity in time of $\bar \q$.
  Let us denote by $[\cdot ]$ the jump function w.r.t. time, i.e. $[\varphi](t) = \varphi(t+) - \varphi(t-)$.
  \revD{%
    Since functions in $\BV(I;X)$ for any Banach space $X$, possess well-defined onesided limits, 
    see \cite[Propositions 2.1 \& 2.2]{heida_topologies_2019}, we obtain from \eqref{eq:control_representation} 
    that $\zz_\alpha(t \pm) \in \Vz$ is well defined for any $t \in I$.
    Hence, by continuity of $\Pad\colon L^2(\Om)^d \to L^2(\Om)^d$ we also have
    $[\bar \q] = [\Pad(\zz_\alpha)] \in L^2(\Om)^d$ for all $t \in I$.}
  By distinguishing the different cases, it is straightforward to verify,
  that for any $\vv \in C(\bar I;L^2(\Om)^d) + \BV(I;L^2(\Om)^d)$ it holds
  \begin{equation*}
    0 \le 
    \Oprod{[\Pad(\vv)],[\Pad(\vv)]}
    \le \Oprod{[\Pad(\vv)],[\vv]}
    \le \Oprod{[\vv],[\vv]} \quad \text{for all } t \in I.
  \end{equation*}
  Applying this chain of inequalities to $\vv = \zz_\alpha$, using $\bar \q = \Pad(\zz_\alpha)$
  and the continuity in time of $\hat \zz$, we obtain
  \begin{equation}\label{eq:positive_jump}
    0 \le \|[\bar \q]\|_{L^2(\Om)}^2
    \le \Oprod{[\bar \q],[\zz_\alpha]}
    = \frac{[\tilde \mu]}{\alpha} \Oprod{[\bar \q],\ww}.
  \end{equation}
%
  For all $t \in I$ satisfying $[\tilde \mu](t) = 0$, this immediately shows $[\bar \q](t) = \oo$.
  Thus let us assume now, that there exists $t^* \in I$ with $[\tilde \mu](t^*) \neq 0$.
  As $\bar \mu \ge 0$, there holds for $t_1 < t_2$ due to 
  \cite[Theorem 4.17, Proposition 4.24]{appell_bounded_2013}
  \begin{equation*}
    0 \le \Ipair{\chi_{[t_1,t_2]},\bar \mu} = \int_{[t_1,t_2]} 1 \d \tilde \mu(t)
    = - \int_{[t_1,t_2]} \tilde \mu \ \d 1(t) + \tilde \mu(t_2) - \tilde \mu(t_1)
    = \tilde \mu(t_2) - \tilde \mu(t_1),
  \end{equation*}
  i.e. $\tilde \mu$ is monotonically increasing and therefore, it holds $[\tilde \mu] \ge 0$.
  Moreover it holds 
  \begin{equation}\label{eq:lagrange_multiplier_singleton}
  \begin{aligned}
    \bar \mu(\{t^*\}) 
    & = \lim_{\varepsilon \to 0} \bar \mu((t^*-\varepsilon,t^*+\varepsilon))
    = \lim_{\varepsilon \to 0} \Ipair{\chi_{(t^*-\varepsilon,t^*+\varepsilon)},\bar \mu}
    = \lim_{\varepsilon \to 0} \int_{(t^*-\varepsilon,t^*+\varepsilon)} 1 \d \tilde \mu\\
    & = \lim_{\varepsilon \to 0}-  \int_{(t^*-\varepsilon,t^*+\varepsilon)} \tilde \mu \ \d 1 
      + \tilde \mu(t^* + \varepsilon) - \tilde \mu(t^* - \varepsilon)
    = [\tilde \mu](t^*).
  \end{aligned}
  \end{equation}
  Thus if $[\tilde \mu](t^*) \neq 0$, then $[\tilde \mu](t^*) > 0$ and $t^* \in \supp(\bar \mu)$,
  yielding that the state constraint is active in
  $t^*$, i.e. $\Gw(\bar \uu(t^*)) = \beta$. 
  As $\ww \in \V$, due to \eqref{eq:weak_state}, it holds
  \begin{equation*}
    \Oprod{\bar \q, \ww} \revision{-} \Oprod{\nabla \bar \uu,\nabla \ww} = \Oprod{\partial_t \bar \uu,\ww} 
    = \partial_t \Gw(\bar \uu) \quad \text{a.e. in } I.
  \end{equation*}
  Since $\bar \q \in C(\bar I;\revD{L^2(\Om)^d}) + \BV(I;\revD{H^1(\Om)^d})$ and
  $\nabla \bar \uu \in C(\bar I; \V)$,
  this identity shows that
  $\partial_t \Gw(\bar \uu) \in L^2(I)$ has representant with well defined onesided limits, which we shall 
  denote by the same symbol, and it holds 
  \begin{equation}\label{eq:jump_opt_control_bound_new}
    [\partial_t \Gw(\bar \uu)](t^*) 
    = \Oprod{[\bar \q](t^*),\ww} 
    \ge \frac{\alpha}{[\tilde \mu](t^*)}\|[\bar \q](t^*)\|_{L^2(\Om)}^2 
    \ge 0,
  \end{equation}
  where the last \revD{inequality} holds due to \eqref{eq:positive_jump} and $[\tilde \mu](t^*) > 0$.
  Using $\Gw(\uu)(t) = \int_0^t \partial_t \Gw(\uu)(s) \ \d s$, it is straightforward to check that 
  \begin{equation*}
    \lim_{h \to 0+} \dfrac{\Gw(\bar \uu)(t\pm h)-\Gw(\uu)(t)}{h} 
    = \lim_{h\to 0+} \frac{1}{h} \int_t^{t\pm h} \partial_t \Gw(\bar \uu)(s) \ \d s 
    = \pm \partial_t \Gw(\bar \uu)(t\pm).
  \end{equation*}
  i.e. the onesided limits of $\partial_t \Gw(\bar \uu)$ correspond to the directional derivatives of 
  $\Gw(\bar \uu)$.
  As $\Gw(\bar \uu)(t^*) = \beta$, $t^*$ is a local maximum of $\Gw(\bar \uu)$, yielding
  \begin{equation*}
    \pm \partial_t \Gw(\bar \uu)(t^* \pm) \le 0 \Rightarrow [\partial_t \Gw(\bar \uu)](t^*) \le 0.
  \end{equation*}
  Combining this with \eqref{eq:jump_opt_control_bound_new} yields $\|[\bar \q](t^*)\|_{L^2(\Om)}^2 = 0$, 
  which shows $\bar \q \in C(\revD{\bar I};L^2(\Om)^d)$. 
  If additionally $\ww \in \Vt$, with the same arguments we can show $\bar \q \in C(\revD{\bar I};H^1(\Om)^d)$.
  %
  \end{proof}
    \begin{remark}
      In the proof of the above theorem, we can under some circumstances also show continuity in time of 
      $\bar \zz$. In fact, the only obstacle preventing this result, is the application of the projection
      operator onto the control constraints $\Pad(\cdot)$. It is in general possible, that at some 
      point in time $\hat t$, $\bar \zz$ is discontinuous. In such cases, the spacial support of the jump
      is contained in the set of points $x \in \Omega$, where $-\frac{1}{\alpha} \bar \zz(\hat t \pm)$ 
      lie outside the set of admissible controls $[\q_a,\q_b]$.
      As the jump of $\bar \zz$ is a scalar multiple of $\ww$, this gives a compatibility condition
      on $\supp(\ww)$ and the set where the control constraints are active.
      As the active set of the control constraints cannot be known a priori, such a condition on $\bar \q$
      is not straightforward to verify. Instead we present in the next two corollaries two 
      sets of assumptions on the data, where we can obtain improved regularity of $\bar \zz$ and $\bar \mu$.
    \end{remark}
    \begin{corollary}
    Let $(\bar \q,\bar \uu,\bar \zz,\bar \mu)$ satisfy the first order necessary optimality conditions
    \eqref{eq: optimal state}-\eqref{eq: Variational inequality}.
    Let further $\q_{a,i} = -\infty$, $\q_{b,i} = + \infty$ for $i=1,...,d$, and $\ww \in \V$.
    Then $\bar \zz \in C(\revD{\bar I};\Vz)$ and $\bar \mu(\{t\}) = 0$ for all 
    $t \in I$, i.e. $\bar \mu$ does not contain
    any Dirac contributions.
    If additionally $\ww \in \Vt$, then $\bar \zz \in C(\revD{\bar I};\V)$.
    \end{corollary}
    \begin{proof}
      This is a direct consequence of the results of \Cref{thm:optimal_regularity_higher_smoothness}, as 
      in this case, the identity $\bar \q = - \frac{1}{\alpha} \bar \zz$ holds.
    \end{proof}
  \begin{corollary}
    Let $(\bar \q,\bar \uu,\bar \zz,\bar \mu)$ satisfy the first order necessary optimality conditions
    \eqref{eq: optimal state}-\eqref{eq: Variational inequality}.
    Let further $\q_a < \bf{0} < \q_b$, $\uu_d \in L^s(I;L^2(\Om)^d)$ for some $s > 2$, and $\ww \in \Vt$
    satisfy
    \begin{equation*}
      \exists x^* \in \partial \Om \text{ s.th. } \forall \varepsilon >0 
      \text{ it holds }  \supp(\ww) \cap B_{\varepsilon}(x^*) \neq \emptyset.
    \end{equation*}
    Then $\bar \zz \in C(\revD{\bar I};\V)$ and $\bar \mu(\{t\}) = 0$ for all $t \in I$, i.e. $\bar \mu$ does not contain
    any Dirac contributions.
  \end{corollary}
  \begin{proof}
    As $\uu_d \in L^s(I;L^2(\Om)^d)$ for some $s>2$, as in the beginning of the proof of
    \Cref{thm:optimal_regularity_higher_smoothness}, we obtain
    \begin{equation*}
      \hat \zz \in W^{1,s}(I;\Vz) \cap L^s(I;\Vt) \hookrightarrow C(\revD{\bar I};C(\bar \Om)),
    \end{equation*}
    where the last embedding holds due to \cite[Corollary 8]{simon_compact_1986}.
    As further $\ww \in \Vt \hookrightarrow C(\bar \Om)$,
    for any $t^* \in I$ it thus holds $\bar \zz(t^* +),\bar \zz(t^*-) \in C(\bar \Omega)$.
    As $\bar \zz(t^* +)(x^*) = \bar \zz(t^*-)(x^*) = \oo$, and $\q_a < \oo < \q_b$, there exists 
    $\delta > 0$, such that
    \begin{equation*}
      - \frac{1}{\alpha} \bar \zz(t^*+)(x), - \frac{1}{\alpha} \bar \zz(t^*-)(x) \in (\q_a,\q_b)
      \quad \text{for all } x \in B_\delta(x^*)\cap \bar \Omega.
    \end{equation*}
    Thus $[\bar \q](t^*)|_{B_\delta(x^*)} = [- \frac{1}{\alpha} \bar \zz](t^*)|_{B_\delta(x^*)}$.
    As $\supp(\ww) \cap B_\delta(x^*) \neq \emptyset$, there exists an open subset 
    $\omega \subset B_\delta(x^*)$, such that $\ww(x) \neq 0$ for all $x \in \omega$.
    In the end, we obtain, using the continuity in time of $\bar \q$ shown in 
    \Cref{thm:optimal_regularity_higher_smoothness}:
    \begin{equation*}
      0 = \|[\bar \q](t^*)\|_{L^2(\Om)}^2 \ge \int_{\omega} [\bar \q](t^*)^2 \d x = 
      \int_\omega \left(\left[ - \frac{1}{\alpha} \bar \zz\right](t^*)\right)^2 \d x
      = \int_\omega \frac{[\tilde \mu](t^*)^2}{\alpha^2} |\ww|^2 \d x \ge 0.
    \end{equation*}
    As $|\ww| > 0$ for all $x \in \omega$, this shows $[\tilde \mu](t^*) = 0$.
    As $t^*$ was arbitrary, this concludes the proof.
  \end{proof}

\section{Finite Element Approximation of the State Equation}\label{sec:discretization}
  
\subsection{Spatial discretization}
Let $\{\Th\}$\label{glos:triangulation} be a 
family of triangulations of $\bar \Omega$, consisting of closed simplices, where we denote by 
$h$\label{glos:h} the maximum mesh-size. Let $\X_h \subset H^1_0(\Omega)^d$\label{glos:fe_space_velocity} and
$M_h \subset L^2_0(\Omega)$\label{glos:fe_space_pressure} be a pair of compatible finite element spaces, i.e.,
\revC{there holds a} uniform discrete inf-sup condition,
\begin{equation}\label{chap02:eq:discrete_infsup}
    \sup_{\vv_h \in \X_h}\frac{\Oprod{q_h,\nabla \cdot \vv_h} }{\twonorm{\nabla \vv_h}} \geq \gamma \|{q_h}\|_{L^2(\Omega)} \quad \forall q_h \in M_h, 
\end{equation}%
with a constant $\gamma >0$ independent of $h$. 
We shall work under the assumption that the discrete spaces have the following approximative properties.
\begin{assumption}\label{ass:interpolation_operators}
  There exist interpolation operators $i_h\colon H^2(\Omega)^d \cap H^1_0(\Omega)^d \to \X_h$
  and $r_h\colon L^2(\Omega) \to M_h$, such that
  \begin{align*}
    \|\nabla(\vv - i_h \vv)\|_{L^2(\Omega)} &\le c h \|\nabla^2 \vv\|_{L^2(\Omega)}
                                        && \hspace{-30mm} \forall \vv \in H^2(\Omega)^2 \cap H^1_0(\Omega)^2,\\
    \|q - r_h q\|_{L^2(\Omega)} &\le c h \|\nabla q\|_{L^2(\Omega)} 
                                && \hspace{-30mm} \forall q \in H^1(\Omega).
  \end{align*}
\end{assumption}
This assumption is valid, for example, for Taylor-Hood and MINI finite elements on shape regular meshes, see \cite[Assumption 7.2]{Behringer_Leykekhman_Vexler_2023}.
We define the space of discretely divergence-free vector fields $\Vh$ as
\begin{equation}\label{eq:discretely_divergence_Xh}
    \Vh = \Set{ \vv_h \in \X_h \colon (\nabla \cdot \vv_h, q_h) =0 \quad \forall q_h \in M_h}.
\end{equation}
  While on a computational level, especially for the examples presented in \Cref{sec:num_results}, we work
  with a discrete velocity-pressure formulation, in our theoretical analysis we will always use the 
  equivalent formulation in discretely divergence free spaces, in order to shorten notation.
  One exception is the following stationary Stokes problem 
  of finding for some given $f \in H^{-1}(\Om)^d$
  a solution $(\uu,p) \in H^1_0(\Om)^d \times L^2_0(\Om)$ to
\begin{equation}\label{eq:stationary_stokes}
    \Oprod{\nabla \uu,\nabla \vv} - \Oprod{\nabla \cdot \vv,p} 
    + \Oprod{\nabla \cdot \uu,q} = \Oprod{\ff,\vv} \quad \forall (\vv,q) \in H^1_0(\Om)^d \times L^2_0(\Om).
\end{equation}
Its discrete approximation in velocity-pressure formulation reads:
Find $(\uu_{h},p_{h}) \in \X_h \times M_h$ satisfying
\begin{equation}\label{eq:stationary_stokes_discrete}
    \Oprod{\nabla \uu_{h},\nabla \vv_{h}} - \Oprod{\nabla \cdot \vv_{h},p_{h}} 
    + \Oprod{\nabla \cdot \uu_{h},q_{h}} = \Oprod{\ff,\vv_{h}} \quad \forall (\vv_{h},q_{h}) \in \X_h \times M_h.
\end{equation}
The above discrete system can be interpreted as a Stokes Ritz projection: given 
$(\uu,p) \in H^1_0(\Om)^d \times L^2_0(\Om)$, find 
$(R_h^S(\uu,p),R_h^{S,p}(\uu,p)):=(\uu_{h},p_{h}) \in \X_h \times M_h$, satisfying
\begin{equation}\label{eq:stationary_Stokes_projection}
    \Oprod{\nabla (\uu-\uu_{h}),\nabla \vv_{h}} - \Oprod{\nabla \cdot \vv_{h},(p-p_{h})} 
    + \Oprod{\nabla \cdot (\uu-\uu_{h}),q_{h}} = 0 \quad \forall (\vv_{h},q_{h}) \in \X_h \times M_h.
\end{equation}
  Note that, if $\uu \in \V$, then it holds $R_h^S(\uu,p) \in \Vh$.
  Further, the Stokes Ritz projection satisfies the following stability, 
  see \cite[Theorem 5.2.1]{boffi_mixed_2013}.
  \begin{equation*}
    \|\nabla R_h^S(\uu,p)\|_{L^2(\Om)} + \|R_h^{S,p}(\uu,p)\|_{L^2(\Om)} 
    \le C (\|\nabla \uu\|_{L^2(\Om)} + \|p\|_{L^2(\Om)}).
  \end{equation*}
  Let us recall the following error estimates for the stationary discrete Stokes problem:
  See \cite[Theorem 5.25]{boffi_mixed_2013},
  \cite[Theorems 4.21, 4.25, 4.28]{john_finite_2016}, 
  \cite[Theorems 53.17 \& 53.19]{ern_finite_2021} or
  \cite[Chapter II, Theorems 1.8 \& 1.9]{girault_finite_1986}.
  \begin{theorem}\label{thm:stokes_ritz}
    Let $(\uu,p)$ and $(\uu_h,p_h)$ denote the solutions to the continuous and discrete 
    stationary Stokes problems \eqref{eq:stationary_stokes} and \eqref{eq:stationary_stokes_discrete} 
    respectively. Then there holds the estimate
    \begin{equation*}
      \|\uu - \uu_h\|_{H^1(\Om)} + \|p - p_h\|_{L^2(\Om)}
      \le C h (\|\uu\|_{H^2(\Om)} + \|p\|_{H^1(\Om)}).
    \end{equation*}
    If $\Omega$ is convex, there further holds the estimate
    \begin{equation*}
      \|\uu - \uu_h\|_{L^2(\Om)} 
      \le C h (\|\uu -\uu_h\|_{H^1(\Om)} + \|p-p_h\|_{L^2(\Om)}).
    \end{equation*}
  \end{theorem}
  Note, that while \cite[Theorem 5.5.6]{boffi_mixed_2013} contains formal results, on how to derive
    error estimates also for $p$ in a weaker norm, e.g. $H^{-1}(\Omega)$, the argument requires 
    $H^2$ regularity results for the compressible Stokes equations.
    As the corresponding results in \cite{kellogg_regularity_1976,dauge_stationary_1989} require an additional
    decaying condition for the compressibility data, this makes derivation of weaker error estimates for the 
    pressure complicated.

\subsection{Temporal discretization: the discontinuous Galerkin method}
\label{chap:IS:section:discontG}

In this section we introduce  the discontinuous Galerkin method for the time discretization of the transient Stokes equations, a similar method was considered, e.g., in \cite{Chrysafinos_Walkington_2010}.
For that, we partition $I = (0,T]$ into subintervals $I_m = (t_{m-1},t_m]$ of length ${k}_m = t_m - t_{m-1}$, where $0= t_0<t_1<\dots <t_{M-1}<t_M = T$. The maximal and minimal time steps are denoted by ${k}=\max_m {k}_m$ and ${k}_{\min}=\min_m{k}_m$, respectively.  The time partition fulfills the following assumptions:
\begin{enumerate}
  \item There are constants $C,\revC{\theta} >0$ independent of \revC{${k}$} such that
	\begin{equation}
          {k}_{\min} \geq C {k}^{\revC{\theta}}. 
	\end{equation}
      \item There is a constant $\kappa >0$ independent of \revC{${k}$} such that for all\\$m=1, 2, \dots, M-1$
	\begin{equation}
	    \kappa^{-1} \leq \frac{{k}_m}{{k}_{m+1}} \leq \kappa.
	\end{equation}
    \item It holds ${k} \leq \frac{T}{4}$.
\end{enumerate}
For a given Banach space $\mathcal{B}$, we define the semi-discrete space $X_{k}^0(\mathcal{B})$ of piecewise
constant functions in time as
\begin{equation}
    X_{k}^0(\mathcal{B}) = \Set{ \vv_{k} \in L^2(I; \mathcal{B}) \colon \vv_{k}\vert_{I_m} \equiv \vv_m
    \text{ for some } \vv_m \in \mathcal{B}, m = 1,2, \dots, M}.
\end{equation}
We use the following standard notation for a function $\uu \in X_{k}^r(\mathcal{B})$
\revB{%
  to denote one-sided limits and jumps at the time nodes
}
\begin{equation}
\uu_m^+ = \lim_{\varepsilon \rightarrow 0^+}\uu(t_m+\varepsilon), \quad \uu_m^- = \lim_{\varepsilon \rightarrow 0^+}\uu(t_m-\varepsilon),\quad [\uu]_m = \uu_m^+-\uu_m^-.
\end{equation}
We define the bilinear form ${B}(\cdot,\cdot)$ by
\[
  {B}(\uu,\vv) = \revC{\sum_{m=1}^M (\partial_t\uu, \vv )_{I_m\times \Omega}} +
  (\nabla \uu, \nabla \vv)_{I\times \Omega}
	+ \sum_{m=2}^M ([\uu]_{m-1},\vv^+_{m-1})_{\Omega} + (\uu_0^+,\vv_0^+)_{\Omega}.
\]
With this bilinear form, we define the fully discrete approximation for the transient Stokes problem on the 
discretely divergence free space $X_{k}^0(\Vh)$:
\begin{equation}\label{eq:fully_discrete_div_free}
	\uu_{kh} \in X_{k}^0(\Vh) \;:\; {B}(\uu_{kh},\vv_{kh}) = (\q,\vv_{kh})_{I \times \Omega}
  \quad \forall \vv_{kh} \in X_{k}^0(\Vh).
\end{equation}
The unique solution to this system is stable, as the following theorem summarizes:
\begin{theorem}\label{thm:stability_discrete_solution_stokes}
  Let $\q \in L^2(I;H^{-1}(\Om)^d)+L^1(I;L^2(\Om)^d)$. Then there exists a unique solution 
  $\uu_{kh} \in X^0_k(\Vh)$ of \eqref{eq:fully_discrete_div_free}, satsifying
  \begin{equation*}
    \|\uu_{kh}\|_{L^2(I;H^1(\Om))} + \|\uu_{kh}\|_{L^\infty(I;L^2(\Om))} 
    \le C \|\q\|_{L^2(I;H^{-1}(\Om))+L^1(I;L^2(\Om))},
  \end{equation*}
  \revision{where the constant $C>0$ is independent of $k,h$.}
\end{theorem}
\begin{proof}
  As \eqref{eq:fully_discrete_div_free} poses a square system of linear equations in finite dimensions,
  it suffices to show uniqueness. This is a standard argument, making use of a discrete Gronwall Lemma,
  see e.g. \cite[Theorem 4.13]{vexler_error_2024} for a proof focusing especially on the 
  $L^1(I;L^2(\Om)^d)$ right hand side case.
\end{proof}
\begin{remark}\label{rem:stability_discrete_dual}
  Rearranging terms in the definition of the bilinear form gives the following dual representation of $B(\cdot,\cdot)$
  \begin{equation*}
    B(\uu,\vv) = 
  \revC{- \sum_{m=1}^M (\uu, \partial_t \vv )_{I_m\times \Omega}}
  + (\nabla \uu, \nabla \vv)_{I\times \Omega}
  - \sum_{m=1}^{M-1} (\uu_m^-,[\vv]_m)_{\Omega} + (\uu_M^-,\vv_M^-)_{\Omega}.
  \end{equation*}
  With the same arguments as above, for given $\mathbf{g} \in L^2(I;H^{-1}(\Om)^d) + L^1(I;L^2(\Om)^d)$, 
  solutions $\zz_{kh} \in X^0_k(\Vh)$ to the discrete dual equation
  \begin{equation*}
    B(\vv_{kh},\zz_{kh}) = \IOprod{\mathbf{g},\vv_{kh}} \quad \forall \vv_{kh} \in X_k^0(\Vh) 
  \end{equation*}
  exist, are unique and satisfy the stability
  \begin{equation*}
    \|\zz_{kh}\|_{L^2(I;H^1(\Om))} + \|\zz_{kh}\|_{L^\infty(I;L^2(\Om))} 
    \le C \|\mathbf{g}\|_{L^2(I;H^{-1}(\Om))+L^1(I;L^2(\Om))}.
  \end{equation*}
\end{remark}

\subsection{Best approximation type fully discrete error estimate  for the Stokes problem in $L^\infty(I;L^2(\Om)^d)$ norm}
In our recent paper \cite{Behringer_Leykekhman_Vexler_2023},
we have established a best approximation type error estimate for the Stokes problem in the 
$L^\infty(I;L^2(\Om)^d)$ norm. From this more general result, we obtain in the case of homogeneous initial 
data the following result, see \cite[Corollary 6.4]{Behringer_Leykekhman_Vexler_2023}.
\begin{theorem}\label{thm: fully discrete best approximation}
  \revB{Let $\q \in L^s(I;L^2(\Om)^d)$ for some $s>1$ and let $\uu\in \W$ be the weak solution to
    \eqref{eq:weak_state} with associated pressure $p$ in the sense of \eqref{eq:associated_pressure}.
Let $\uu_{kh}\in X_{k}^0(\Vh)$ be the fully discrete Galerkin solution to \eqref{eq:fully_discrete_div_free}.
}
  Then there exists a constant $C$ independent of $k$ and $h$, such that
  for any $\chi\in X_{k}^0(\Vh)$ there holds
\[
    \|\uu-\uu_{kh}\|_{L^\infty(I;L^2(\Om))} \le C\ell_k \left(\|\uu-\chi\|_{L^\infty(I;L^2(\Om))}+\|\uu-R^S_h(\uu,p)\|_{L^\infty(I;L^2(\Om))}\right),
\]
where $\ell_k=\lk$ and $R^S_h(\uu,p)$ is the 
stationary finite element 
Stokes projection introduced in \eqref{eq:stationary_Stokes_projection}.
\end{theorem}
\revC{Using the error estimates for the stationary Stokes Ritz projection of \Cref{thm:stokes_ritz}, 
  in \cite[Theorem 7.4]{Behringer_Leykekhman_Vexler_2023} the following estimate in terms of explicit orders
  of convergence was shown.
}

\begin{corollary}\label{cor: error estimate for convex}
If in addition to assumptions of Theorem \ref{thm: fully discrete best approximation}, the domain $\Omega$
is convex, and  $\revB{\q} \in L^\infty(I;L^2(\Omega)^d)$,
then  there exists a constant $C$ independent of $k$ and $h$ such that
$$
\|\uu-\uu_{kh}\|_{L^\infty(I;L^2(\Om))}\le C\ell_k^2(h^2+k)\|\revB{\q}\|_{L^\infty(I;L^2(\Omega))}.
$$
\end{corollary}
The above results are valid in $L^2(I; L^2(\Om)^d)$ norm as well. However, using the energy and duality
arguments it is possible to show the corresponding results log-free and with 
less regularity assumptions on the data (cf. \cite[Theorem 11 \& Theorem 13]{Leykekhman_Vexler_2023}).

\begin{theorem}\label{thm: fully discrete best approximation L2}
\revB{Let $\q \in L^2(I;L^2(\Om)^d)$ and } let $\uu\in \W$ be the weak solution to \eqref{eq:weak_state} 
\revB{with associated pressure $p$ in the sense of \eqref{eq:associated_pressure}.}
Let $\uu_{kh}\in X_{k}^0(\Vh)$ be the fully discrete Galerkin solution to \eqref{eq:fully_discrete_div_free}.
Then there exists a constant $C$, independent of $k$ and $h$, such that for any
$\chi\in X_{k}^0(\Vh)$, there holds
\begin{subequations}
\begin{equation}\label{eq: best approximation L2}
\|\uu-\uu_{kh}\|_{L^2(I; L^2(\Om))}\le C \left(\|\uu-\chi\|_{L^2(I; L^2(\Om))}+\|\uu-R^S_h(\uu,p)\|_{L^2(I; L^2(\Om))}+\|\uu-\pi_k\uu\|_{L^2(I; L^2(\Om))}\right),
\end{equation}
and
\begin{equation}\label{eq: best approximation H1}
\begin{aligned}
\|\na(\uu-\uu_{kh})\|_{L^2(I; L^2(\Om))} \le C \bigg(\|\na(\uu-\chi)\|_{L^2(I; L^2(\Om))}&+\|\na(\uu-R^S_h(\uu,p))\|_{L^2(I; L^2(\Om))}\\
&+\|\na(\uu-\pi_k\uu)\|_{L^2(I; L^2(\Om))}\bigg),
\end{aligned}
\end{equation}
\end{subequations}
where  $R^S_h(\uu,p)$ is the stationary finite element Stokes projection defined in 
\eqref{eq:stationary_Stokes_projection} and $\pi_k$ is the time projection \revB{onto $X^0_k(\V)$,} with 
$\pi_k v|_{I_m} \revB{= v(t_m^-)}$ for $m=1,2,\dots,M$.
\end{theorem}

\begin{corollary}\label{cor: error estimate for convex L2}
If in addition to assumptions of Theorem \ref{thm: fully discrete best approximation L2},
the domain $\Omega$ is convex, then  there exists a constant $C$ independent of $k$ and $h$ such that
$$
\|\uu-\uu_{kh}\|_{\LtwoLtwo}\le C(h^2+k)\|\revB{\q}\|_{\LtwoLtwo}. 
$$
\end{corollary}

\section{Variational Discretization of the Optimal Control Problem}\label{sec:var_disc_opt_control}
In this section we consider the optimal control problem subject to the fully discretized Stokes equations.
We consider a variational discretization for the controls, i.e., do not fix a finite dimensional 
approximation of the control space yet,
\revision{cf., \cite{deckelnick_variational_2011,hinze_variational_2005}}
The problem reads
\begin{subequations}
\begin{equation}\label{eq:discrete_optimal}
\text{Minimize }\; J(\q_{kh},\uu_{kh}) = \frac{1}{2} \norm{\uu_{kh} - \uu_d}^2_{\LtwoLtwo}+ \frac{\alpha}{2} \norm{\q_{kh}}^2_{\LtwoLtwo}
\end{equation}
over all $\q_{kh} \in \Q_{ad}$, $\uu_{kh} \in X^0_{k}(\Vh)$, subject to
\begin{equation}\label{eq:discrete_state0}
B(\uu_{kh},\vv_{kh}) = ( \q_{kh}, \vv_{kh})_{I\times\Om} \quad \text{for all}\quad \vv_{kh} \in X^0_{k}(\Vh)
\end{equation}
\begin{equation}\label{eq:discrete_constraint}
  G_\ww(\uu_{kh})\mid_{I_{m}}\le \beta\quad\text{for}\quad m=1,2,\dots,M.
\end{equation}
 \end{subequations}
Following the structure of \Cref{sect:cont_opt_control} and using 
\Cref{thm:stability_discrete_solution_stokes},
we introduce the discrete analog to the control state map,
\begin{equation}
  S_{kh}\colon L^2(I;H^{-1}(\Om)^d) + L^1(I;L^2(\Om)^d) \to X^0_k(\Vh),\quad \q \mapsto \uu_{kh} = \uu_{kh}(\q)
  \text{ solving \eqref{eq:fully_discrete_div_free}}.
\end{equation}
The finitely many state constraints we describe with the help of the continuous linear operator 
$\mathcal{G}_{kh}\colon \Q_{ad} \to \R^M$
with $(\mathcal{G}_{kh}(\q))_m:=G_{\bf w}\circ S_{kh} (\q) \mid_{I_m}$ for $m=1,2,\dots,M$. Using the set 
$$
\mathcal{K}_{kh}:=\{\vv \in\mathbb{R}^M\ \mid \ \vv _m\le b,\ m=1,2,\dots,M \},
$$
we can rewrite the problem \eqref{eq:discrete_optimal}-\eqref{eq:discrete_constraint} in the reduced form:
\begin{equation}\label{discrete optimal problem reduced}
  \text{Minimize }\ j_{kh}(\q_{kh}):=J(\q_{kh},S_{kh}(\q_{kh})) \quad \text{over} \quad \q_{kh}\in \Q_{ad}\quad 
\text{subject to}\quad 
\mathcal{G}_{kh}(\q_{kh})\in \mathcal{K}_{kh}.
\end{equation}
Before discussing wellposedness and optimality conditions of this discrete problem, we shall show, that the 
Slater assumption on the continuous level carries over to the discrete problem.
\revC{%
  As we achieve this with the finite element error estimates presented in \Cref{sec:discretization},
  we need to impose a rather weak coupling condition between $k$ and $h$, allowing us to deduce convergence
  in the result of \Cref{cor: error estimate for convex}.
  Throughout the remainder of this work, we thus work under the following assumption.
  \begin{assumption}\label{ass:khCoupling}
    There exists a function $\Phi\colon (0,1) \to (0,\infty)$ with 
    $\lim_{h\to 0} \Phi(h) = 0$, such that the discretization parameters $k$ and $h$ satisfy
    \begin{equation*}
      \left|\ln\left(\frac{T}{k}\right)\right| h \le \Phi(h).
    \end{equation*}
  \end{assumption}
  \begin{remark}
    This assumption is valid, e.g. if there exists a constant $C> 0$ such that 
    $\left|\ln \left(\frac{T}{k}\right)\right| h |\ln h| \le C$. As the choice of the term $|\ln h|$ in such a condition can be
    made arbitrarily weak, we have chosen to work under the more general formulation of \Cref{ass:khCoupling}.
  \end{remark}
}
 \begin{lemma}\label{lemm:discrete_Slater}
  There exists $h_0 > 0$ such that for any \revC{$h\le h_0$ and $k$ satisfying \Cref{ass:khCoupling}},
  the Slater point 
  $\tilde \q \in \Q_{ad}$ from \Cref{Slater} satisfies the following discrete Slater condition
  \begin{equation*}
    \Gw(\uu_{kh}(\tilde q)) < \beta \revC{\qquad \text{for all } t \in \bar I.}
  \end{equation*}
\end{lemma}
\begin{proof}
  Using that $G_{\bf w}(\uu(\tilde{\q}))<\beta$ \revC{in $\bar I$},
  by the Slater condition Assumption \ref{Slater} there exists $\delta>0$ 
  such that $G_{\bf w}(\uu(\tilde{\q}))\le \beta-\delta$. 
  For arbitrary $\hat \q \in C^\infty(I\times \Om)^d$ it holds due to triangle inequality
  \begin{equation*}
    \|\uu(\tilde \q) - \uu_{kh}(\tilde \q)\|_{L^\infty(I;L^2(\Om))}
    \le 
    \|\uu(\tilde \q) - \uu_{kh}(\tilde \q) - \uu(\hat \q) + \uu_{kh}(\hat \q)\|_{L^\infty(I;L^2(\Om))}+
    \|\uu(\hat \q) - \uu_{kh}(\hat \q)\|_{L^\infty(I;L^2(\Om))}.
  \end{equation*}
  Using the continuous and fully discrete stability results of the state equations, presented in 
  \Cref{thm:solvability_state_equation} and \Cref{thm:stability_discrete_solution_stokes},
  as well as the error estimate \Cref{cor: error estimate for convex} for the problem with right hand side 
  $\hat \q$, we obtain
  \begin{equation*}
    \|\uu(\tilde \q) - \uu_{kh}(\tilde \q)\|_{L^\infty(I;L^2(\Om))}
    \le 
    C \|\tilde \q - \hat \q\|_{L^2(I\times \Om)} +
    C \ell_k^2 (k + h^2) \|\hat \q\|_{L^\infty(I;L^2(\Om))}.
  \end{equation*}
  \revC{For any} $\varepsilon > 0$, due to the density of $C^\infty(I\times \Om)^d$ in ${L^2(I;L^2(\Om)^d)}$,
  we can find $\hat \q_\varepsilon$ such that 
  $C \|\tilde \q - \hat \q_\varepsilon\|_{L^2(I \times \Om))} < \frac{\varepsilon}{2}$. Moreover, 
  \revC{for $h\le h_0$ sufficiently small, and $k$ satisfying \Cref{ass:khCoupling}}, it also holds 
  $C \ell_k^2 (k + h^2) \|\hat \q_\varepsilon\|_{L^\infty(I;L^2(\Om))} 
  < \frac{\varepsilon}{2}$.
  Thus in total 
  $\|\uu(\tilde \q) - \uu_{kh}(\tilde \q)\|_{L^\infty(I;L^2(\Om))} < \varepsilon$.
  \revC{Choosing $\varepsilon$ small enough, such that $\|\ww\|_{L^2(\Om)}\varepsilon < \delta$, we obtain}
  \begin{equation}\label{eq: discrete Slater}
  G(\uu_{kh}(\tilde{\q}))=G(\uu(\tilde{\q}))+G(\uu_{kh}(\tilde{\q})-\uu(\tilde{\q}))
  \le \beta-\delta+\|\ww\|_{L^2(\Om)}\|\uu(\tilde{\q})-\uu_{kh}(\tilde{\q})\|_{L^\infty(I;L^2(\Om))}<\beta.
  \end{equation}
%
\end{proof}

\begin{theorem}
  \revC{Let $k$ and $h$ satisfy \Cref{ass:khCoupling}, and let $h$ be small enough. Then}
    there exists a unique solution $(\bar \q_{kh},\bar \uu_{kh})$ to the optimal control problem
\eqref{eq:discrete_optimal}-\eqref{eq:discrete_constraint}.
\end{theorem}
\begin{proof}
  As \Cref{lemm:discrete_Slater} shows feasibility of $(\tilde \q,\uu_{kh}(\tilde \q))$ under the given 
  assumptions, the existence proof follows the same steps as the one of \Cref{thm:existence_optimal_control}
  on the continuous level.
\end{proof}


\begin{theorem}[Discrete first order optimality system]\label{discrete first order optimality system}
  A control $\bar{\q}_{kh}\in \Q_{ad}$ and the associated state $\bar{\uu}_{kh}=\uu_{kh}(\bar{\q}_{kh}) \in X^0_{k}(\Vh)$ is \revC{the} optimal solution to the problem
  \eqref{eq:discrete_optimal}-\eqref{eq:discrete_constraint}
  if and only if there exists an adjoint state $\bar{\zz}_{kh}  \in X^0_{k}(\Vh)$ and a Lagrange multiplier $\bar{\mu}_{kh}\in \revB{L^1(I)}$ 
that satisfy:\\
\begin{subequations}
{Discrete state equation}
\begin{equation}\label{discrete state}
B(\bar{\uu}_{kh},\vv_{kh}) = ( \bar{\q}_{kh}, \vv_{kh})_{I\times\Om}, \quad \forall \vv_{kh} \in X^0_{k}(\Vh);
\end{equation}
{Discrete state constraint and complementarity conditions}
\begin{equation}\label{discrete complementary}
  \revB{\Gw(\bar \uu_{kh})|_{I_m} \le \beta, \ m=1,...,M}, \quad 
  \bar{\mu}_{kh}\geq 0 \quad\text{and}\quad \langle \bar{\mu}_{kh},\revision{\beta}-\revB{\Gw(\bar \uu_{kh})} \rangle=0;
\end{equation}
{Discrete adjoint equation}
\begin{equation}\label{discrete adjoint}
B(\vv_{kh},\bar{\zz}_{kh}) = ( \bar{\uu}_{kh}-\uu_d+{\bar{\mu}_{kh}}\ww, \vv_{kh})_{I\times\Om} \quad \forall \vv_{kh} \in X^0_{k}(\Vh);
\end{equation}
{Discrete variational inequality}
\begin{equation}\label{discrete variational}
\left(\alpha\bar{\q}_{kh}+\bar{\zz}_{kh}, \delta \q-\bar{\q}_{kh}\right)_{I\times \Om}\geq 0 \quad  \forall \delta \q\in \Q_{ad} \quad \Leftrightarrow \quad \bar{\q}_{kh} =P_{[\q_a,\q_b]}\left(-\frac{1}{\alpha}\bar{\zz}_{kh}\right).
\end{equation}
\end{subequations}
  Furthermore, there exist $\bar \mu_{kh}^m \in \R_{\ge0}$, $m=1,2,\dots,M$, such that the discrete Lagrange
  multiplier $\bar{\mu}_{kh}\in L^1(I)$ satisfies the expression 
  \begin{equation}\label{discrete Lagrange}
    \bar \mu_{kh} = \sum_{m=1}^M \frac{\bar \mu_{kh}^m}{k_m} \chi_{I_m},
  \end{equation}
  where $\chi_{I_m}$ denotes the characteristic function of the interval $I_m$.
\end{theorem}
\begin{proof}
    In \Cref{lemm:discrete_Slater}, we have shown, that \revC{under \Cref{ass:khCoupling} and for small
    enough $h$,} there holds $\mathcal{G}_{kh}(\tilde{\q})\in \operatorname{int}(\mathcal{K}_{kh})$.
    Similarly to the proof of Theorem \ref{thm:continuous first order optimality} we obtain that the optimality
    of $\bar{q}_{kh}$ is equivalent to the existence of a Lagrange multiplier 
    $(\bar{\mu}_{kh}^m)_{m=1}^M\in \mathbb{R}^M_{\ge 0}$ and the adjoint state 
    $\bar{\zz}_{kh}\in X^0_{k}(\Vh)$ satisfying \eqref{discrete complementary}, \eqref{discrete adjoint} and
    \eqref{discrete variational}. Finally, by the construction given in \eqref{discrete Lagrange},
    $\bar{\mu}_{kh}$ is an element of \revB{$L^1(I)$}.
\end{proof}

\begin{remark} \label{rem: norm mu_kh}
  Notice, that from the definition \eqref{discrete Lagrange} and using that \revB{$\bar \mu_{kh} \ge 0$,
  it holds}
$$
\revB{\|\bar \mu_{kh}\|_{L^1(I)} =} \|\bar{\mu}_{kh}\|_{C(\bar{I})^*}=\langle \bar{\mu}_{kh},1 \rangle
\revB{= \sum_{m=1}^M \bar \mu_{kh}^m.}
$$
\end{remark}

\begin{remark}
  We would like to point out that although the state $\bar{\uu}_{kh}$ and the adjoint $\bar{\zz}_{kh}$ are 
  fully discrete, the corresponding  control  $\bar{\q}_{kh}\in \Q_{ad}$ is piecewise constant in time via
  \eqref{discrete variational}, but not necessary piecewise \revB{polynomial in space with respect to the 
  given mesh, due to the projection onto $[\q_a,\q_b]$.}
\end{remark}

With the optimality conditions established, we now show the following 
stability of optimal solutions to the discrete problem subject to different discretization levels.


\begin{lemma}\label{lem: fully discrete stability}
  \revC{Under \Cref{ass:khCoupling} and for $h$ small enough, there exists $C>0$ independent of $k$, $h$, 
  such that} the
  optimal control $\bar{\q}_{kh}\in \Q_{ad}$  of the variationally discretized problem
  \eqref{eq:discrete_optimal} - \eqref{eq:discrete_constraint}, together with its corresponding state 
  $\bar{\uu}_{kh}\in X^0_{k}(\Vh)$ and corresponding multiplier $\bar{\mu}_{kh}\in \revB{L^1(\bar I)}$
  satisfy the bound
$$
\|\bar \q_{kh}\|_{L^\infty(I;L^2(\Om))}+
\|\bar{\uu}_{kh}\|_{\LtwoLtwo}+\|\bar{\mu}_{kh}\|_{L^1(\bar I)}\le C.
$$
\end{lemma}
\begin{proof}
  \revB{%
    Due to the feasibility of $\tilde \q$, shown in \Cref{lemm:discrete_Slater}, it holds
  }
$$
\begin{aligned}
J(\bar{\q}_{kh},\bar{\uu}_{kh})\le J(\tilde{\q},\uu_{kh}(\tilde{\q}))
&=\frac{1}{2} \norm{\uu_{kh}(\tilde{\q}) - \uu_d}^2_{\LtwoLtwo} + \frac{\alpha}{2} \norm{\tilde{\q}}^2_{L^2(I)}
\le C(T,\tilde{\q}),
\end{aligned}
$$
\revB{%
  where due to \Cref{thm:stability_discrete_solution_stokes}, this bound is independent of $k$ and $h$.
  This results in
}
\begin{equation}\label{eq:discrete_control_state_bounded}
\|\bar{\q}_{kh}\|_{\LtwoLtwo}+\|\bar{\uu}_{kh}\|_{\LtwoLtwo}\le C.
\end{equation}
Let us define $\pp=\frac{1}{2}\bar{\q}+\frac{1}{2}\tilde{\q}$.
By definition $\pp\in \Q_{ad}$ and thus by \eqref{discrete variational} it holds
$
\left(\alpha\bar{\q}_{kh}+\bar{\zz}_{kh},  \pp-\bar{\q}_{kh}\right)_{I\times \Om}\geq 0.
$
This yields
    \begin{equation}\label{eq:discrete_optcond_expanded}
\begin{aligned}
0 &\le \alpha\left(\bar{\q}_{kh},  \pp-\bar{\q}_{kh}\right)_{I\times \Om}+\left(\bar{\zz}_{kh},  \pp-\bar{\q}_{kh}\right)_{I\times \Om}\\
&=\alpha\left(\bar{\q}_{kh},  \pp-\bar{\q}_{kh}\right)_{I\times \Om}+B(\uu_{kh}(\pp)-\bar{\uu}_{kh},\bar{\zz}_{kh})\\
&=\alpha\left(\bar{\q}_{kh},  \pp-\bar{\q}_{kh}\right)_{I\times \Om}+(\uu_{kh}(\pp)-\bar{\uu}_{kh},\bar{\uu}_{kh}-\uu_d)_{I\times \Om}+\langle \bar{\mu}_{kh},G_\ww({\uu}_{kh}(\pp))-G_\ww(\bar{\uu}_{kh}) \rangle,
\end{aligned}
\end{equation}
\revB{%
  where we can bound the first two terms by \eqref{eq:discrete_control_state_bounded} and obtain
  \begin{equation}\label{eq:discrete_stability_eq1}
    0 \le C + \langle \bar{\mu}_{kh},G_\ww({\uu}_{kh}(\pp))-G_\ww(\bar{\uu}_{kh}) \rangle.
  \end{equation}
}
\revC{For $\pp$, using \Cref{ass:khCoupling}, we can follow a similar argument as in the proof of 
\Cref{lemm:discrete_Slater}, in order to obtain}
$$
\begin{aligned}
G_{\bf w}(\uu_{kh}({\pp}))&=G_{\bf w}(\uu_{kh}({\pp})-\uu({\pp}))+G_{\bf w}(\uu({\pp}))\\
&=G_{\bf w}(\uu_{kh}({\pp})-\uu({\pp}))+\frac{1}{2}G_{\bf w}(\bar{\uu})+\frac{1}{2}G_{\bf w}(\uu(\tilde{\q}))\\
&\le \| \ww\|_{L^2(\Om)}\|\uu_{kh}({\pp})-\uu({\pp})\|_{L^\infty(I; L^2(\Om))}+\frac{1}{2}G_{\bf w}(\bar{\uu})+\frac{1}{2}G_{\bf w}(\uu(\tilde{\q}))\\
&\le\frac{1}{4}\delta+ \frac{1}{2}\beta+\frac{1}{2}\beta-\frac{1}{2}\delta =\beta-\frac{1}{4}\delta.
\end{aligned}
$$
Inserting this into \eqref{eq:discrete_stability_eq1}, yields together with $\bar \mu_{kh} \ge 0$
\revB{and the complementarity conditions \eqref{discrete complementary}}:
$$
0 \le C+\langle \bar{\mu}_{kh},\revision{\beta}-G_\ww(\bar{\uu}_{kh}) \rangle-\frac{\delta}{4}\langle \bar{\mu}_{kh},1 \rangle
=C-\frac{\delta}{4}\langle \bar{\mu}_{kh},1 \rangle.
$$
Thus, again using $\bar{\mu}_{kh}\geq 0$ and Remark \ref{rem: norm mu_kh}, results in 
\begin{equation}\label{eq:discrete_multiplier_bounded}
\|\bar{\mu}_{kh}\|_{L^1(\bar I)}=\langle \bar{\mu}_{kh},1 \rangle\le C.
\end{equation}
    Combining \eqref{eq:discrete_control_state_bounded} and \eqref{eq:discrete_multiplier_bounded}
    with \revB{\Cref{rem:stability_discrete_dual}}
    yields the boundedness of $\bar \zz_{kh}$ in $L^\infty(I;L^2(\Om)^d)$.
    By the representation $\bar \q_{kh} = P_{[\q_a,\q_b]} \left( - \frac{1}{\alpha} \bar \zz_{kh} \right)$,
    this shows $\|\bar \q_{kh}\|_{L^\infty(I;L^2(\Om))} \le C$, which concludes the proof.
\end{proof}

\begin{theorem}\label{thm:err_est_var_discrete_minprob}
  \revC{Let \Cref{ass:khCoupling} hold and let $h$ be sufficiently small.}
  Let $(\bar \q,\bar \uu)$ and $(\bar \q_{kh},\bar \uu_{kh})$ be the unique solutions to the continuous and 
  variationally discretized optimal control problems
  \eqref{eq:optimal_problem} -\eqref{eq:state_constraint} and
  \eqref{eq:discrete_optimal}-\eqref{eq:discrete_constraint} respectively.
  Then there exists a constant $C>0$, such that it holds
\[
\sqrt{\alpha}\norm{\bar{\q} - \bar{\q}_{kh}}_{\LtwoLtwo}+\norm{\bar{\uu} - \bar{\uu}_{kh}}_{\LtwoLtwo} \le C \, \ell_k (k^{\frac{1}{2}}+h) , \quad \ell_k = \lk.
\]
\end{theorem}

\begin{proof}
Choosing $\delta\q=\bar{\q}_{kh}$ in \eqref{eq: Variational inequality} and $\delta\q=\bar{\q}$ in \eqref{discrete variational}, results in
\begin{equation}
\left(\alpha\bar{\q}+\bar{\zz}, \bar{\q}_{kh}-\bar{\q}\right)_{I\times \Om}\geq 0\quad \text{and}\quad \left(\alpha\bar{\q}_{kh}+\bar{\zz}_{kh}, \bar{\q}-\bar{\q}_{kh}\right)_{I\times \Om}\geq 0.
\end{equation}
Adding these two inequalities results in 
\begin{equation}
\alpha\|\bar{\q}_{kh}-\bar{\q}\|^2_{\LtwoLtwo}\le (\bar{\zz},\bar{\q}_{kh}-\bar{\q})_{I\times \Om}+ (\bar{\zz}_{kh},\bar{\q}-\bar{\q}_{kh})_{I\times \Om}:=I_1+I_2.
\end{equation} 
We estimate the two terms separately. 

{\bf{Estimate for $I_1$.}}
\revB{%
Using the weak formulations \eqref{eq:weak_state} and \eqref{eq:weak_formulation_adjoint} of the 
continuous state and adjoint equations \eqref{eq: optimal state} \& \eqref{eq: adjoint state}, respectively,
we have
}
$$
\begin{aligned}
  I_1 &= \IOpair{\pa_t (\uu(\bar{\q}_{kh})-\bar{\uu}),\bar \zz}
  +\IOprod{\na (\uu(\bar{\q}_{kh})-\bar{\uu}),\na \bar \zz}\\
&= (\bar{\uu}-\uu_d,\uu(\bar{\q}_{kh})-\bar{\uu})_{I \times \Omega}+\langle G_{\bf w}(\uu(\bar{\q}_{kh})-\bar{\uu}),\bar{\mu}\rangle.
\end{aligned}
$$
Introducing the pointwise projection onto the state constraint
\begin{equation}\label{eq:projection_stateconstraint}
P_\beta v= P_{\beta}v(t) := \max\{v(t),\beta\},\quad \forall t\in \bar{I},
\end{equation}
the last term  can be estimated as 
$$
\begin{aligned}
\langle G_{\bf w}(\uu(\bar{\q}_{kh})-\bar{\uu}),\bar{\mu}\rangle 
&=\langle G_{\bf w}(\uu(\bar{\q}_{kh}))-P_{\beta}G_{\bf w}(\uu(\bar{\q}_{kh})),\bar{\mu}\rangle
+\langle P_{\beta}G_{\bf w}(\uu(\bar{\q}_{kh}))-G_{\bf w}(\bar{\uu}),\bar{\mu}\rangle\\
&\le \langle G_{\bf w}(\uu(\bar{\q}_{kh}))-P_{\beta}G_{\bf w}(\uu(\bar{\q}_{kh})),\bar{\mu}\rangle
+\langle {\beta} -G_{\bf w}(\bar{\uu}),\bar{\mu}\rangle,
\end{aligned}
$$
\revB{%
  where we have used, that due to $\bar \mu \ge 0$, it holds 
  $\langle P_{\beta} \Gw(\uu(\bar \q_{kh})),\bar \mu\rangle \le \langle {\beta},\bar \mu\rangle$.
  Using the \revC{complementarity} condition \eqref{eq: Complementary conditions}, it holds 
  $\langle {\beta} -G_{\bf w}(\bar{\uu}),\bar{\mu}\rangle = 0$, hence we have 
  \begin{equation*}
    \langle G_{\bf w}(\uu(\bar{\q}_{kh})-\bar{\uu}),\bar{\mu}\rangle 
    \le \langle G_{\bf w}(\uu(\bar{\q}_{kh}))-P_{\beta}G_{\bf w}(\uu(\bar{\q}_{kh})),\bar{\mu}\rangle.
  \end{equation*}
}
Now using that 
\begin{equation}\label{eq: Lipschitz Pb}
|P_{\beta}v-P_{\beta}u|\le |v-u|,
\end{equation}
by the triangle inequality and using that $G_{\bf w}(\bar{\uu}_{kh})\le {\beta}$, we obtain
$$
\begin{aligned}
    \langle G_{\bf w}(\uu(\bar{\q}_{kh}))-P_{\beta}G_{\bf w}(\uu(\bar{\q}_{kh})),\bar{\mu}\rangle \le & 
\| G_{\bf w}(\uu(\bar{\q}_{kh}))-P_{\beta}G_{\bf w}(\uu(\bar{\q}_{kh})\|_{L^\infty(I)}\|\bar{\mu}\|_{C(\bar{I})^*}\\
\le &
\| G_{\bf w}(\uu(\bar{\q}_{kh}))-G_{\bf w}(\bar{\uu}_{kh})\|_{L^\infty(I)}\|\bar{\mu}\|_{C(\bar{I})^*}\\
&+\| P_{\beta} G_{\bf w}(\bar{\uu}_{kh})-P_{\beta}G_{\bf w}(\uu(\bar{\q}_{kh})\|_{L^\infty(I)}\|\bar{\mu}\|_{C(\bar{I})^*}\\
\le & 2\|{\bf w}\|_{L^2(\Om)}\| \revision{\uu_{kh}}(\bar{\q}_{kh})-\uu(\bar{\q}_{kh})\|_{L^\infty(I ;L^2(\Omega))}\|\bar{\mu}\|_{C(\bar{I})^*}\\
\le & C \ell_k^2(h^2+k)\|\bar{\q}_{kh}\|_{L^\infty(I;L^2(\Omega))}\|{\bf w}\|_{L^2(\Om)}\|\bar{\mu}\|_{C(\bar{I})^*}\\
\le & C \ell_k^2(h^2+k),
\end{aligned}
$$
where in the last two steps we used Corollary \ref{cor: error estimate for convex} and Lemma \ref{lem: fully discrete stability}.
Thus,
$$
I_1 \le C \ell_k^2(h^2+k)+(\bar{\uu}-\uu_d,\uu(\bar{\q}_{kh})-\bar{\uu})_{I \times \Omega}.
$$
{\bf{Estimate for $I_2$.}}
Similarly, using the fully discrete state and adjoint equations \eqref{discrete state} and \eqref{discrete adjoint}, respectively, we have
$$
\begin{aligned}
I_2 &= B(\uu_{kh}(\bar{\q})-\bar{\uu}_{kh},\bar{\zz}_{kh})\\
&= (\bar{\uu}_{kh}-\uu_d,\uu_{kh}(\bar{\q})-\bar{\uu}_{kh})_{I \times \Omega}+\langle G_{\bf w}(\uu_{kh}(\bar{\q})-\bar{\uu}_{kh}),\bar{\mu}_{kh}\rangle.
\end{aligned}
$$
Using the projection $P_{\beta}$ defined in \eqref{eq:projection_stateconstraint},
the last term in $I_2$ can be estimated as 
$$
\begin{aligned}
\langle G_{\bf w}(\uu_{kh}(\bar{\q})-\bar{\uu}_{kh}),\bar{\mu}\rangle 
&=\langle G_{\bf w}(\uu_{kh}(\bar{\q}))-P_{\beta}G_{\bf w}(\uu_{kh}(\bar{\q})),\bar{\mu}_{kh}\rangle
+\langle P_{\beta}G_{\bf w}(\uu_{kh}(\bar{\q}))-G_{\bf w}(\bar{\uu}_{kh}),\bar{\mu}_{kh}\rangle\\
&\le \langle G_{\bf w}(\uu_{kh}(\bar{\q}))-P_{\beta}G_{\bf w}(\uu_{kh}(\bar{\q})),\bar{\mu}_{kh}\rangle+\langle {\beta} -G_{\bf w}(\bar{\uu}_{kh}),\bar{\mu}_{kh}\rangle,
\end{aligned}
$$
  where we have used, that due to $\bar \mu_{kh} \ge 0$, it holds 
  $\langle P_{\beta} \Gw(\uu_{kh}(\bar \q)),\bar \mu_{kh}\rangle \le \langle {\beta},\bar \mu_{kh}\rangle$.
  Using the \revC{complementarity} condition \eqref{discrete complementary}, it holds 
  $\langle {\beta} - \Gw(\bar{\uu}_{kh}),\bar{\mu}_{kh}\rangle = 0$, hence we have 
  \begin{equation*}
    \langle G_{\bf w}(\uu_{kh}(\bar{\q})-\bar{\uu}_{kh}),\bar{\mu}\rangle 
    \le \langle G_{\bf w}(\uu_{kh}(\bar{\q}))-P_{\beta}G_{\bf w}(\uu_{kh}(\bar{\q})),\bar{\mu}_{kh}\rangle.
  \end{equation*}
Using \eqref{eq: Lipschitz Pb}, the triangle inequality and using that $G_{\bf w}(\bar{\uu}_{kh})\le {\beta}$,
we obtain
$$
\begin{aligned}
 \langle G_{\bf w}(\uu_{kh}(\bar{\q}))-P_{\beta}G_{\bf w}(\uu_{kh}(\bar{\q}),\bar{\mu}_{kh}\rangle \le & 
\| G_{\bf w}(\uu_{kh}(\bar{\q}))-P_{\beta}G_{\bf w}(\uu_{kh}(\bar{\q})\|_{L^\infty(I)}\|\bar{\mu}_{kh}\|_{L^1(\bar I)}\\
\le &
\| G_{\bf w}(\uu_{kh}(\bar{\q}))-G_{\bf w}(\bar{\uu})\|_{L^\infty(I)}\|\bar{\mu}_{kh}\|_{L^1(\bar I)}\\
&+\| P_{\beta} G_{\bf w}(\bar{\uu})-P_{\beta}G_{\bf w}(\uu_{kh}(\bar{\q})\|_{L^\infty(I)}\|\bar{\mu}_{kh}\|_{L^1(\bar I)}\\
\le & 2\|{\bf w}\|_{L^2(\Om)}\| \revision{\uu}(\bar{\q})-(\uu_{kh}(\bar{\q})\|_{L^\infty(I ;L^2(\Omega))}\|\bar{\mu}_{kh}\|_{L^1(\bar I)}\\
\le & C \ell_k^2(h^2+k)\|\bar{\q}\|_{L^\infty(I;L^2(\Omega))}\|{\bf w}\|_{L^2(\Om)}\|\bar{\mu}_{kh}\|_{L^1(\bar I)}\\
\le & C \ell_k^2(h^2+k),
\end{aligned}
$$
where in the last two steps we used Corollary \ref{cor: error estimate for convex} and Lemma \ref{lem: fully discrete stability}.
Thus,
$$
I_2 \le C \ell_k^2(h^2+k)+(\bar{\uu}_{kh}-\uu_d,\uu_{kh}(\bar{\q})-\bar{\uu}_{kh})_{I \times \Omega}.
$$
Combining the estimates for $I_1$ and $I_2$ and using that
$$
\begin{aligned}
(\bar{\uu}-\uu_d,\uu(\bar{\q}_{kh})&-\bar{\uu})_{I \times \Omega}+(\bar{\uu}_{kh}-\uu_d,\uu_{kh}(\bar{\q})-\bar{\uu}_{kh})_{I \times \Omega}\\
=&(\bar{\uu}-\uu_d,\bar{\uu}_{kh}-\bar{\uu})_{I \times \Omega}+(\bar{\uu}-\uu_d,\uu(\bar{\q}_{kh})-\bar{\uu}_{kh})_{I \times \Omega}\\
&+(\bar{\uu}_{kh}-\uu_d,\uu_{kh}(\bar{\q})-\bar{\uu})_{I \times \Omega}+(\bar{\uu}_{kh}-\uu_d,\bar{\uu}-\bar{\uu}_{kh})_{I \times \Omega}\\
=&-\norm{\bar{\uu} - \bar{\uu}_{kh}}^2_{\LtwoLtwo}+(\bar{\uu}-\uu_d,\uu(\bar{\q}_{kh})-\bar{\uu}_{kh})_{I \times \Omega}+(\bar{\uu}_{kh}-\uu_d,\uu_{kh}(\bar{\q})-\bar{\uu})_{I \times \Omega}
\end{aligned}
$$
by using Corollary \ref{cor: error estimate for convex L2}, we obtain
$$
\begin{aligned}
\alpha & \norm{\bar{\q} - \bar{\q}_{kh}}^2_{\LtwoLtwo}+\norm{\bar{\uu} - \bar{\uu}_{kh}}^2_{\LtwoLtwo} \\
&\le  C\ell^2_k (k+h^2)+(\bar{\uu}-\uu_d,\uu(\bar{\q}_{kh})-\bar{\uu}_{kh})_{I \times \Omega}+(\bar{\uu}_{kh}-\uu_d,\uu_{kh}(\bar{\q})-\bar{\uu})_{I \times \Omega}\\
&\le C\ell^2_k (k+h^2)+\|\bar{\uu}-\uu_d\|_{\LtwoLtwo}\|\uu(\bar{\q}_{kh})-\revision{\uu_{kh}}(\bar{\q}_{kh})\|_{\LtwoLtwo}\\
&\quad +\|\bar{\uu}_{kh}-\uu_d\|_{\LtwoLtwo}\|\uu_{kh}(\bar{\q})-\revision{\uu}(\bar{\q})\|_{\LtwoLtwo}\\
&\le C\ell^2_k (k+h^2)\left[1+\left(\|\bar{\uu}\|_{\LtwoLtwo}+\|\bar{\uu}_{kh}\|_{\LtwoLtwo}+\|\uu_d\|_{\LtwoLtwo}\right)\left(\|\bar{\q}_{kh}\|_{\LtwoLtwo}+\|\bar{\q}\|_{\LtwoLtwo}\right)\right]\\
&\le C\ell^2_k (k+h^2),
\end{aligned}
$$
where in the last step, we used the boundedness of $\|\bar{\uu}_{kh}\|_{\LtwoLtwo}$, $\|\bar{\uu}\|_{\LtwoLtwo}$, $\|\bar{\q}\|_{\LtwoLtwo}$, and $\|\bar{\q}_{kh}\|_{\LtwoLtwo}$ from Theorem \ref{thm: regularity optimal} and Lemma \ref{lem: fully discrete stability}. 
\end{proof}

\section{Full Discretization of the Optimal Control Problem}\label{sec:full_disc_opt_control}

We discretize the control by piecewise constant functions on the same partition as the fully discrete  approximation of the state and adjoint variables. We set
\begin{equation}\label{eq: Q_0}
\Q_0 = \{ \q\in L^2(I;L^2(\Omega)^d):\ \q \mid_{I_m\times K}\in \mathbb{P}_0(I_m; \mathbb{P}_0(K)^d), \ m=1,2,\dots,M,  \ K\in \Th \}.
\end{equation}
We also define the corresponding admissible set
$$
\Q_{0,ad}:= \Q_0\cap \Q_{ad}.
$$
We introduce the projection $\pi_d\colon L^2(I;L^2(\Om)^d) \to \Q_0$, defined by
\begin{equation}\label{eq:discrete_control_projection}
    \IOprod{\q - \pi_d \q,\mathbf{r}} = 0 \qquad \forall \mathbf{r} \in \Q_0,
\end{equation}
\revision{which by definition is stable in $\LtwoLtwo$, i.e., satisfies
  \begin{equation}\label{eq:ltwo_stability_projection}
    \|\pi_d \q\|_{\LtwoLtwo} \le \|\q\|_{\LtwoLtwo}.
  \end{equation}
}
Note that this projection satisfies the explicit formula
\begin{equation*}
  \pi_d \q|_{I_m \times K} = \frac{1}{k_m |K|}\int_{I_m \times K} \q(t,x) \ d(t,x)
  \qquad \text{for all } m=1,...,M, \ K \in \mathcal T_h.
\end{equation*}
Hence it is straightforward to check, that this $L^2$ projection onto piecewise constants is 
stable in $L^\infty(I \times \Om)^d$ and $L^\infty(I;L^2(\Omega)^d)$ and there holds
\begin{equation}\label{eq:linfty_stability_projection}
  \|\pi_d \q\|_{L^\infty(I \times \Om)} \le \|\q\|_{L^\infty(I \times \Om)} \quad \text{and} \quad
  \|\pi_d \q\|_{L^\infty(I;L^2(\Om))} \le \|\q\|_{L^\infty(I;L^2(\Om))}.
\end{equation}
Further, we have $\pi_d (\Q_{ad}) \subset \Q_{0,ad}$.
We can now formulate the fully discrete optimal control problem, which reads
\begin{subequations}
  \begin{equation}\label{eq:minprob_fully_discrete}
\text{Minimize }\; J(\q_{\sigma},\uu_{\sigma}) = \frac{1}{2} \norm{\uu_{\sigma} - \uu_d}^2_{\LtwoLtwo} + \frac{\alpha}{2} \norm{\q_{\sigma}}^2_{\LtwoLtwo}
  \end{equation}
subject to $(\q_{\sigma},\uu_{\sigma}) \in \Q_{0,ad}\times X^0_{k}(\Vh)$, satisfying 
\begin{equation}\label{eq:stateeq_fully_discrete}
B(\uu_{\sigma},\vv_{kh}) = ( \q_{\sigma}, \vv_{kh})_{I\times\Om} \quad \text{for all}\quad \vv_{kh} \in X^0_{k}(\Vh)
\end{equation}
 and
 \begin{equation}\label{eq:stateconstraint_fully_discrete}
  G_\ww(\uu_{\sigma})\mid_{I_{m}}\le \beta\quad\text{for}\quad m=1,2,\dots,M.
\end{equation}
 \end{subequations}

  The following lemma guarantees that also for the fully discrete optimal control problem, there exist
  feasible controls such that the associated fully discrete state strictly satisfies the state constraint.
 \begin{lemma}\label{lemm:fully_discrete_Slater}
   \revC{Let \Cref{ass:khCoupling} be satisfied and let $h$ be sufficiently small. Then}
   the projection 
   \revC{$\pi_d \tilde \q \in \Q_{0,ad}$} of the 
  Slater point $\tilde \q \in \Q_{ad}$ from \Cref{Slater} satisfies the following discrete Slater condition
  \begin{equation*}
    \Gw(\uu_{kh}(\pi_d \tilde q)) < \beta \revC{\qquad \text{for all } t \in \bar I}.
  \end{equation*}
\end{lemma}
\begin{proof}
  From \Cref{lemm:discrete_Slater}, we know, that there exists $\delta > 0$ such that 
  $\Gw(\uu_{kh}(\tilde q)) < \beta - \delta$. As the discrete solution operator $S_{kh}$ is linear
  and continuous from $L^2(I;L^2(\Om)^d) \to L^\infty(I;L^2(\Om)^d)$, we have
  \begin{equation*}
    \Gw(\uu_{kh}(\pi_d \tilde \q)) = 
    \Gw(\uu_{kh}(\pi_d \tilde \q - \tilde \q))
    + \Gw(\uu_{kh}(\tilde \q))
    < C\|\pi_d \tilde \q - \tilde \q\|_{\LtwoLtwo} + \beta -\delta.
  \end{equation*}
  As \revC{\Cref{ass:khCoupling} guarantees $k \to 0$ as $h \to 0$, it holds}
  $\|\pi_d \tilde \q - \tilde \q\|_{L^2(I; L^2(\Om))} \to 0$ for $h \to 0$. This implies that for 
  $h$ small enough, we have $C\|\pi_d \tilde \q - \tilde \q\|_{\LtwoLtwo} \le \delta$
  and as a consequence $\Gw(\uu_{kh}(\pi_d \tilde \q)) < \beta$.
\end{proof}

\begin{theorem}
  \revC{Let \Cref{ass:khCoupling} be satisfied and let $h$ be sufficiently small. Then }
    there exists a unique solution $(\bar \q_\sigma, \bar \uu_\sigma)$ to the fully discrete optimal control
    problem \eqref{eq:minprob_fully_discrete}-\eqref{eq:stateconstraint_fully_discrete}.
\end{theorem}
\begin{proof}
  As \Cref{lemm:fully_discrete_Slater} shows feasibility of $(\pi_d \tilde \q,\uu_{kh}(\pi_d \tilde \q))$
  \revC{under the given assumptions,}
  the existence proof follows the same steps as the one of \Cref{thm:existence_optimal_control}
  on the continuous level.
\end{proof}

Similar to \Cref{sec:var_disc_opt_control}, we can rewrite the problem \eqref{eq:minprob_fully_discrete}-\eqref{eq:stateconstraint_fully_discrete} in the reduced form
\begin{equation}\label{eq:fully_discrete_optimal_control_reduced}
\text{Minimize }\ j_{kh}(\q_\sigma)=J(\q_\sigma,S_{kh}(\q_\sigma)) \quad \text{over} \quad \q_\sigma\in \Q_{0,ad}\quad 
\text{subject to}\quad 
\mathcal{G}_{kh}(\q_\sigma)\in \mathcal{K}_{kh}.
\end{equation}
Note that compared to the the variationally discretized optimal control problem
\eqref{discrete optimal problem reduced}, only the control space has changed.


\begin{theorem}[First order optimality conditions for discretized controls]
  \label{fully fully discrete first order optimality system}
A control $\bar{\q}_{\sigma}\in \Q_{0,ad}$ and the associated state
$\bar{\uu}_{\sigma}=\uu_{kh}(\bar{\q}_{\sigma}) \in X^0_{k}(\Vh)$ is \revC{the} optimal  
solution to the problem \eqref{eq:minprob_fully_discrete}-\eqref{eq:stateconstraint_fully_discrete} 
if and only if 
there exists an adjoint state $\bar{\zz}_{\sigma}  \in X^0_{k}(\Vh)$ and a Lagrange multiplier 
$\bar{\mu}_{\sigma}\in \revB{L^1(I)}$ that satisfy:\\
\begin{subequations}
{Discrete state equation}
\begin{equation}\label{discrete state sigma}
B(\bar{\uu}_{\sigma},\vv_{kh}) = ( \bar{\q}_{\sigma}, \vv_{kh})_{I\times\Om}, \quad \forall \vv_{kh} \in X^0_{k}(\Vh);
\end{equation}
{Discrete state constraint and complementarity conditions}
\begin{equation}\label{discrete complementary sigma}
  \revB{\Gw(\bar \uu_\sigma)|_{I_m} \le \beta, \ m=1,...,M, \quad }
  \bar{\mu}_{\sigma}\geq 0 
  \quad\text{and}\quad \langle \bar{\mu}_{\sigma},\beta-\revB{\Gw(\bar \uu_\sigma)} \rangle=0;
\end{equation}
{Discrete adjoint equation}
\begin{equation}\label{discrete adjoint sigma}
B(\vv_{kh},\bar{\zz}_{\sigma}) = ( \bar{\uu}_{\sigma}-\uu_d+{\bar{\mu}_{\sigma}}\ww, \vv_{kh})_{I\times\Om} \quad \forall \vv_{kh} \in X^0_{k}(\Vh);
\end{equation}
{Discrete variational inequality}
\begin{equation}\label{discrete variational sigma}
\left(\alpha\bar{\q}_{\sigma}+\bar{\zz}_{\sigma}, \delta \q-\bar{\q}_{\sigma}\right)_{I\times \Om}\geq 0 \quad  \forall \delta \q\in \Q_{0,ad} \quad \Leftrightarrow \quad \bar{\q}_{\sigma} =P_{[\q_a,\q_b]}\left(-\frac{1}{\alpha}\bar{\zz}_{\sigma}\right).
\end{equation}
\end{subequations}
\revB{%
  Furthermore, there exist $\bar \mu_{\sigma}^m \in \R_{\ge 0}$, $m=1,2,\dots,M$,
  such that the discrete Lagrange
  multiplier $\bar{\mu}_{\sigma}\in L^1(I)$ satisfies the expression 
  \begin{equation}\label{discrete Lagrange sigma}
    \bar \mu_{\sigma} = \sum_{m=1}^M \frac{\bar \mu_{\sigma}^m}{k_m} \chi_{I_m},
  \end{equation}
  where $\chi_{I_m}$ denotes the characteristic function of the interval $I_m$.
}
\end{theorem}
\begin{proof}
The proof is almost identical to the proof of Theorem \ref{discrete first order optimality system}.
\end{proof}

\revB{%
  Again, due to $\bar \mu_\sigma \ge 0$, it holds $\|\bar \mu_\sigma\|_{L^1(I)} = \|\bar \mu_\sigma\|_{(C(\bar I)^*)} = \langle \bar \mu_\sigma, 1 \rangle = \sum_{m=1}^M \bar \mu_\sigma^m$.
}

\begin{lemma}\label{lem: fully discrete stability discrete control}
  \revC{Let \Cref{ass:khCoupling} be satisfied, and $h$ be small enough. Then there exists a constant $C>0$
    independent of $k$ and $h$, such that the fully discrete optimal control $\bar{\q}_{\sigma}\in \Q_{0,ad}$,
    solving \eqref{eq:minprob_fully_discrete}-\eqref{eq:stateconstraint_fully_discrete}, together with its 
    corresponding state $\bar{\uu}_{\sigma}\in X^0_{k}(\Vh)$ and corresponding multiplier 
    $\bar{\mu}_{\sigma}\in L^1(\bar I)$ satisfies the bound}
$$
\|\bar{\q}_{\sigma}\|_{L^\infty(I;L^2(\Om))}
+\|\bar{\uu}_{\sigma}\|_{\LtwoLtwo}+\|\bar{\mu}_{\sigma}\|_{(C(\bar{I}))^*}\le C.
$$
\end{lemma}
\begin{proof}
  By \Cref{lemm:fully_discrete_Slater} under the given assumptions, the fully discrete control
  $\pi_d \tilde \q$ is feasible, and thus it holds 
$$
\begin{aligned}
  J(\bar{\q}_{\sigma},\bar{\uu}_{\sigma})
  &\le J(\pi_d\tilde{\q},\revision{\uu_{kh}}(\pi_d\tilde{\q}))\\
  &=\frac{1}{2} \norm{\revision{\uu_{kh}}(\pi_d\tilde{\q}) - \uu_d}^2_{\LtwoLtwo} + \frac{\alpha}{2} \norm{\pi_d\tilde{\q}}^2_{\LtwoLtwo}\\
    &\le \norm{\revision{\uu_{kh}}(\pi_d \tilde{\q})}^2_{\LtwoLtwo} + 
    \norm{\uu_d}^2_{\LtwoLtwo}+\frac{\alpha}{2}\norm{\tilde{\q}}^2_{\LtwoLtwo} \\
&\le C(T,\tilde{\q},\uu_d),
\end{aligned}
$$
\revision{where in the last step, we have used the discrete stability result for $\uu_{kh}$ from
  \Cref{thm:stability_discrete_solution_stokes} and the stability of $\pi_d$ from 
  \eqref{eq:ltwo_stability_projection}.}
As a result we obtain
$
\|\bar{\q}_{\sigma}\|_{\LtwoLtwo}+\|\bar{\uu}_{\sigma}\|_{\LtwoLtwo}\le C.
$
\revB{%
  The proof of $\|\bar \mu_\sigma\|_{C(\bar I)^*)} \le C$ and $\|\bar \q_\sigma\|_{L^\infty(I;L^2(\Om))} \le C$
  is then accomplished by following the same steps
  as the proof of \Cref{lem: fully discrete stability},  making use of 
  $\Gw(\uu_{kh}({\pi_d\pp}))\le \beta -\frac{1}{4}\delta$ for $h$ sufficiently small.
}
\end{proof}


\begin{theorem}\label{thm:err_est_full_discrete_minprob}
  \revC{Let \Cref{ass:khCoupling} be satisfied and let $h$ be small enough.}
  Let $(\bar \q_{kh},\bar \uu_{kh})$ and $(\bar \q_\sigma,\bar \uu_\sigma)$ denote the optimal solutions
  of the variationally discretized optimal control problem 
  \eqref{eq:discrete_optimal}-\eqref{eq:discrete_constraint} and the fully discretized optimal control
  problem \eqref{eq:minprob_fully_discrete}-\eqref{eq:stateconstraint_fully_discrete}.
  Then there exists a constant $C>0$ such that it holds
\[
\alpha\norm{\bar{\q}_{kh} - \bar{\q}_{\sigma}}^2_{\LtwoLtwo}
+\norm{\bar{\uu}_{kh} - \bar{\uu}_{\sigma}}^2_{\LtwoLtwo} \le Ch^2.
\]
\end{theorem}
\begin{proof}
  Choosing $\delta\q=\bar{\q}_{\sigma} \in \Q_{ad}$ in \eqref{discrete variational} and
  $\delta\q=\pi_d\bar{\q}_{kh} \in \Q_{0,ad}$ in \eqref{discrete variational sigma}, results in
\begin{equation}
\left(\alpha\bar{\q}_{kh}+\bar{\zz}_{kh}, \bar{\q}_{\sigma}-\bar{\q}_{kh}\right)_{I\times \Om}\geq 0\quad \text{and}\quad \left(-\alpha\bar{\q}_{\sigma}-\bar{\zz}_{\sigma}, \bar{\q}_{\sigma}-\pi_d\bar{\q}_{kh}\right)_{I\times \Om}\geq 0.
\end{equation}
Adding these two inequalities, we obtain
\begin{equation}
\alpha\|\bar{\q}_{kh}-\bar{\q}_\sigma\|^2_{\LtwoLtwo}\le (\bar{\zz}_{kh}- \bar{\zz}_{\sigma},\bar{\q}_{\sigma}-\bar{\q}_{kh})_{I\times \Om}+
(\alpha\bar{\q}_{\sigma}+\bar{\zz}_{\sigma}, \pi_d\bar{\q}_{kh}-\bar{\q}_{kh})_{I\times \Om}:=I_1+I_2.
\end{equation} 
We estimate the two terms separately. 

{\bf{Estimate for $I_1$.}}
Using the  discrete state equations \eqref{discrete state},\eqref{discrete state sigma} the corresponding adjoint equations \eqref{discrete adjoint} and \eqref{discrete adjoint sigma}, respectively,  we have
$$
\begin{aligned}
I_1 &= B(\bar{\uu}_{\sigma}-\bar{\uu}_{kh},\bar{\zz}_{kh}- \bar{\zz}_{\sigma})\\
&= (\bar{\uu}_{\sigma}-\bar{\uu}_{kh},\bar{\uu}_{kh}-\bar{\uu}_{\sigma})_{I \times \Omega}+\langle G_{\bf w}(\bar{\uu}_{\sigma})-G_{\bf w}(\bar{\uu}_{kh}),\bar{\mu}_{kh}\rangle-\langle G_{\bf w}(\bar{\uu}_{\sigma})-G_{\bf w}(\bar{\uu}_{kh}),\bar{\mu}_{\sigma}\rangle\\
&\le -\norm{\bar{\uu}_{kh} - \bar{\uu}_{\sigma}}^2_{\LtwoLtwo}+\langle \beta -G_{\bf w}(\bar{\uu}_{kh}),\bar{\mu}_{kh}\rangle+\langle \beta -G_{\bf w}(\bar{\uu}_{\sigma}),\bar{\mu}_{\sigma}\rangle\\
&= -\norm{\bar{\uu}_{kh} - \bar{\uu}_{\sigma}}^2_{\LtwoLtwo}.
\end{aligned}
$$
By the Cauchy-Schwarz inequality and properties of the $L^2$-projection
\begin{equation*}
\begin{aligned}
    I_2 &= \IOprod{\alpha \bar \q_\sigma + \bar \zz_\sigma,\pi_d\bar \q_{kh} - \bar \q_{kh}}\\
    &    = \IOprod{\bar \zz_\sigma - \pi_d \bar \zz_\sigma,\pi_d\bar \q_{kh} - \bar \q_{kh}}\\
    &     \le C h^2 \|\nabla \bar \zz_\sigma\|_{\LtwoLtwo} \|\nabla \bar \q_{kh}\|_{\LtwoLtwo}.
         \end{aligned}
\end{equation*}
Using that $\| \na\bar{\q}_{kh}\|_{\LtwoLtwo}\le \alpha^{-1}\| \na\bar{\zz}_{kh}\|_{\LtwoLtwo}$,
\revision{ which holds due to the projection formula \eqref{discrete variational}
  and the stability of $P_{ad}$ in $H^1(\Om)$,
  see \cite[Theorem A.1]{kinderlehrer_introduction_2000},
  \cite[Theorem 5.8.2]{attouch_variational_2014},}
\revB{%
  and the stability of solutions to the fully discrete dual problem, pointed out in 
  Remark \ref{rem:stability_discrete_dual}, gives
}
$$
I_2\le Ch^2\left(\|\bar{\uu}_{\sigma} - \bar{\uu}_{d}\|_{\LtwoLtwo}+\|\bar{\mu}_{\sigma}\|_{L^1(I)}\|\ww\|_{L^2(\Om)}\right)\left(\|\bar{\uu}_{kh} - \bar{\uu}_{d}\|_{\LtwoLtwo}+\|\bar{\mu}_{kh}\|_{L^1(I)}\|\ww\|_{L^2(\Om)}\right).
$$
Now the boundedness of $\|\bar{\uu}_{\sigma}\|_{\LtwoLtwo}$, $\|\bar{\uu}_{kh}\|_{\LtwoLtwo}$, $\|\bar{\mu}_{\sigma}\|_{L^1(I)}$, and $\|\bar{\mu}_{kh}\|_{L^1(I)}$ from Lemmas \ref{lem: fully discrete stability} and \ref{lem: fully discrete stability discrete control}
finish the proof.
\end{proof}

With this last error estimate, our main result now directly follows from \Cref{thm:err_est_var_discrete_minprob}
and \Cref{thm:err_est_full_discrete_minprob}.
\begin{theorem}[Error estimate for the control]\label{thm:main_result}
  \revC{Let \Cref{ass:khCoupling} be satisfied and let $h$ be small enough.}
Let $\bar \q \in \Q_{ad}$ and $\bar \q_\sigma \in \Q_{0,ad}$ be the solutions to the continuous and fully 
discrete optimal control problems 
\eqref{eq:optimal_problem}-\eqref{eq:state_constraint} and
\eqref{eq:minprob_fully_discrete}-\eqref{eq:stateconstraint_fully_discrete}.
\[
\sqrt{\alpha}\norm{\bar{\q} - \bar{\q}_{\sigma}}_{\LtwoLtwo}+\norm{\bar{\uu} - \bar{\uu}_{\sigma}}_{\LtwoLtwo} \le C \, \ell_k (k^{\frac{1}{2}}+h) , \quad \ell_k = \lk.
\]
\end{theorem}

\begin{remark}
 All of our main results presented in   
 \Cref{thm:err_est_var_discrete_minprob,thm:err_est_full_discrete_minprob,thm:main_result} 
 do not require any additional regularity assumptions on the optimal control $\bar \q$, but 
 work precisely with the regularity that is obtained from the optimality conditions.
 Furthermore, the techniques used in proving these results allow us to avoid 
 strong coupling conditions on $k$ and $h$.
 \revC{We only require the product $\ell_k h$ to converge to zero, in order to
   obtain a Slater condition for the discrete problems, which can be guaranteed by a very mild 
 coupling condition as in \Cref{ass:khCoupling}.}
\end{remark}

\section{Improved regularity}\label{sec:improved_regularity}
Despite having no a priori smoothness regularity for $\bar{\zz}$ and as a result for $\bar{\q}$, 
our main result shows almost $k^{\frac{1}{2}}$ convergence rate for the error of the optimal control.
Similarly to \cite{MeidnerD_RannacherR_VexlerB_2011}, we can use this result to establish improved regularity  for the optimal control.
\begin{theorem}
Let $\bar{\q}\in \Q_{ad}$ be the optimal solution to \eqref{eq:optimal_problem}-\eqref{eq:state_constraint}. Then,
$$
\bar{\q}\in L^2(I;H^1(\Omega)^d)\cap H^s(I;L^2(\Om)^d),\quad \forall s<\frac{1}{2}.
$$ 
Additionally, if $\q_a,\q_b \in \R^d$, i.e., are finite, it holds 
$\bar \q \in \revC{L^\infty(I \times \Om)^d}$.
\end{theorem}
\begin{proof}
The proof is identical to the proof of Theorem 7.1 in
\cite{MeidnerD_RannacherR_VexlerB_2011}.
\end{proof}
\revC{In light of \Cref{thm:optimal_regularity_higher_smoothness} and the imbedding $\BV(I) \hookrightarrow H^s(I)$ for all $s < \frac{1}{2}$, this result is particularly interesting, as it holds without additional 
regularity assumptions on the weight $\ww$.
Note also, that the regulartiy properties of \Cref{thm:optimal_regularity_higher_smoothness} and its 
corrollaries were not used in the derivation of our error estimates.
It remains an open question, how \revC{these results} can be used 
to derive improved error estimates for the optimal control problem directly.
}

\section{Numerical results}\label{sec:num_results}
In the following, we present \revision{three} numerical examples,
studying the orders of convergence for the optimal 
control problem. We first present an example for smooth data. These results can be compared to
the numerical example in \cite{MeidnerD_RannacherR_VexlerB_2011}, where the weight function
$w(x_1,x_2)=\sin{(\pi x_1)}\sin{(\pi x_2)}$ on the unit square was used. As this weight is in
$H^1_0(\Om)\cap H^2(\Om)$ almost first order convergence in time was observed, which we again can see 
for the Stokes optimal control problem.
In the second example, we study an example where less regularity of the data is available, leading to a 
reduced order of convergence.
\revision{Lastly, we consider an example with simultaneous control and state constraints.
  Throughout this section, we will analyze the errors of the full discretization $\bar \q - \bar \q_{\sigma}$.
  For details regarding the 
  implementation of a variational discretization in the presence of control constraints, we refer to
  \cite{hinze_variational_2005}, \cite[Chapter 3.2.5]{hinze_optimization_2009} and the 
  references therein.
}
\subsection{Solution Algorithm}
Equation \eqref{eq:fully_discrete_optimal_control_reduced} can be written in matrix notation as
\begin{equation}
  \min_{\q_\sigma \in \R^N} 
  \frac{1}{2} (S_{kh} M_{kh}^{uq} \q_\sigma - \uu_{d,kh})^T M_{kh}^{uu} (S_{kh} M_{kh}^{uq} 
  \q_\sigma -\uu_{d,kh}) 
  + \frac{1}{2} \alpha \q_\sigma^T M_{kh}^{qq} \q_\sigma \qquad \text{s.t.} \quad
  W_{kh} S_{kh} \revC{M_{kh}^{uq}} \q_\sigma \le \beta,
\end{equation}
where $M_{kh}^{uu}$ is the mass matrix of $X^0_k(\Vh)$, $M_{kh}^{qq}$ is the mass matrix of $\Q_0$, 
\revC{$M_{kh}^{uq}$ is the mass matrix encoding inner products of $X^0_k(\Vh)$ and $\Q_0$ functions,}
$\uu_{d,kh}$ is the $L^2$ projection of $\uu_d$ onto $X^0_k(\Vh)$ and $W_{kh}$ is the matrix that maps 
each $\uu_{kh}$ to the vector $(\Gw(\uu_{kh})|_{I_m})_{m=1,..,M}$.
Note that with a slight abuse of notation, we write $S_{kh}$ for the control to state mapping as well as 
the matrix representing this mapping.
The optimality conditions for this problem read
\begin{align*}
  \alpha M_{kh}^{qq} \q_\sigma + (M_{kh}^{uq})^T S_{kh}^T M_{kh}^{uu} (S_{kh} M_{kh}^{uq} \q_h - \uu_{kh,d})
  + (M_{kh}^{uq})^T S_{kh}^T W_{kh}^T \mu_\sigma &= 0,\\
  W_{kh} S_{kh} M_{kh}^{uq} \q_\sigma \le \beta, \quad \mu_\sigma \ge 0,
  \quad \mu_\sigma^T (W_{kh} S_{kh} M_{kh}^{uq} \q_\sigma - \beta) &=0.
\end{align*}
We solve the above problem with a primal-dual-active-set strategy (PDAS), during which, for each iteration of 
the active set $\mathcal A_n \subset \{1,\dots,M\}$, we solve a symmetric \revD{saddle point} system
\begin{equation*}
  \begin{pmatrix}
    \alpha M_{kh}^{qq} + (M_{kh}^{uq})^T S_{kh}^T M_{kh}^{uu} S_{kh} M_{kh}^{uq}  & (M_{kh}^{uq})^T S_{kh}^T W_{kh}^T O_n\\
    O_n^T W_{kh} S_{kh} M_{kh}^{uq} & 0
  \end{pmatrix}
  \begin{pmatrix}
    \q_\sigma^n \\ \mu_\sigma^n
  \end{pmatrix}
  = \begin{pmatrix}
    (M_{kh}^{uq})^T S_{kh}^T M_{kh}^{uu} \uu_{kh,d}\\ \beta \cdot \mathds{1}_{|\mathcal A_n|}
  \end{pmatrix},
\end{equation*}
where $O_n \in \R^{|\mathcal A_n|\times M}$ is the matrix satisfying 
$O_n \mu_\sigma = (\mu_\sigma)_{\mathcal A_n}$, i.e., selects the active indices.
We solve the linear system with MINRES, using the block diagonal preconditioner
\begin{equation*}
  P_{kh} = \begin{pmatrix}
    \alpha M_{kh}^{qq} & 0 \\ 0 & - \frac{1}{\alpha} W_{kh}S_{kh}M_{kh}^{uq} (M_{kh}^{qq})^{-1} (M_{kh}^{uq})^T S_{kh}^T W_{kh}^T
  \end{pmatrix}.
\end{equation*}
Note that due to the choice of control discretization $M_{kh}^{qq}$ is a diagonal matrix, and if the 
partition of $I$ is uniform, and all mesh elements of $\Omega$ have the same volume, $M_{kh}^{qq}$ is a 
multiple of the identity matrix.
Further note, that if the partition of $I$ is uniform, the matrix $S_{kh}^T W_{kh}^T$ has the structure
\begin{equation*}
  \begin{pmatrix}
  \zz_M & \zz_{M-1} & ... & \zz_2 & \zz_1\\
  0     & \zz_{M} & ... & \zz_3 & \zz_2 \\
  \vdots & \vdots & \ddots & \vdots & \vdots\\
  0 & 0 & ... & \zz_{M} & \zz_{M-1}\\
  0 & 0 & ... & 0 & \zz_M
\end{pmatrix} \in \R^{ M\dim(\Vh) \times M}
\end{equation*}
where each vector $\zz_m \in \R^{\dim(\Vh)}$, $m=1,...,M$ corresponds to the degrees of freedom of
$\zz|_{I_m}$ for the solution $\zz$ to 
\begin{equation*}
  B(\vv_{kh},\zz_{kh}) = \IOprod{\ww \chi_{I_M},\vv_{kh}} \quad \text{for all} \quad \vv_{kh} \in X^0_k(\V_h).
\end{equation*}
Hence, to assemble this matrix only one discrete adjoint problem has to be solved, and a decomposition of 
the preconditioning matrix $P_{kh}$ can be computed in advance and be reused in every iteration of the 
PDAS algorithm.
The discrete solutions of the finite element problems were carried out in FEniCS Version 2019
\cite{LoggEtal2012}, using the MINI Element in space. 

\subsection{Example 1}\label{subsec:examp_1}
For this example, we consider the setting $\Omega = (0,1)^2$, $I = (0,1]$,
\revB{%
choose the regularization parameter $\alpha = 1$ and set the control constraints to
$\q_a = (- \infty,-\infty)^T$, $\q_b = (+ \infty,+\infty)^T$.
}
We construct an analytic test case by considering the functions
\begin{equation*}
  \varphi(t) := \begin{cases}{}
    48 t^2 - 128 t^3 & t \in [0,1/4),\\
    1 & t \in [1/4,3/4],\\
    48 (1-t)^2 - 128 (1-t)^3 & t \in (3/4,1],
  \end{cases}
  \quad
  \yy = \frac{64 \sqrt{2}}{5\sqrt{7}}\begin{pmatrix}
    -\sin(\pi x_1)^4 \cos(\pi x_2) \sin(\pi x_2)^3\\
    \cos(\pi x_1) \sin(\pi x)^3 \sin(\pi x_2)^4
  \end{pmatrix}.
\end{equation*}
Here, $\yy$ has been constructed in such a way, that $\nabla \cdot \yy = 0$ and 
$\|\yy\|_{L^2(\Omega)}=1$.
It was obtained by considering the potential $\rho(x_1,x_2) = \left(\sin(\pi x_1)\sin(\pi x_2)\right)^4$,
and normalizing the vector field $(\partial_{x_2}\rho(x_1,x_2),-\partial_{x_1}\rho(x_1,x_2))^T$.
We choose $\bar{\uu}=\varphi(t)\yy(x_1,x_2)$, $\ww = \yy$ and $\beta \equiv 1$.
It is then straightforward to verify, that $\Gw(\bar{\uu}) \le \beta$ for all $t \in I$ 
and $\Gw(\bar{\uu}) = \beta$ if and only if 
$t \in [1/4,3/4]$. We thus choose the multiplier $\bar \mu = 10^3 \chi_{[1/4,3/4]}(t)$,
which by construction satisfies $\bar \mu \ge 0$ and $\langle \Gw(\bar \uu) -\beta,\bar \mu\rangle = 0$.
We then proceed by choosing $\bar p = 0$ and as a consequence set 
$\bar \q = \partial_t \bar \uu - \Delta \bar \uu$. We obtain $\bar \zz = - \bar \q$ and with the choice
$\bar r = 0$ can fix $\uu_d = \bar \uu + \partial_t \bar \zz + \Delta \bar \zz + \bar \mu \ww$, in such a way,
that the constructed $(\bar \q,\bar \uu)$ satisfy the first order optimality condition for this desired state.
Note that $\vec \psi$ has been chosen sufficiently smooth at the boundary, such that $\Delta \vec\psi|_{\partial \Omega} = 0$ and thus $\bar \q|_{\partial \Omega} = \bar \zz|_{\partial \Omega} = 0$.
The calculation of the analytic solution was verified using the SageMath software \cite{sagemath}.
We discretize this problem with a uniform triangulation of $\Omega$ and a uniform partition of $I$.
To get more insight into the observed orders of convergence, for
\revision{a sequence of discretization levels $\{\sigma_l\} = \{(h_l,k_l)\}$,}
we report the empirical orders of convergence determined by
\begin{align*}
  (\mathrm{EOC}_h)_l &:= \dfrac{\log(\|\bar \q - \bar \q_{\revision{\sigma_l}}\|_{\LtwoLtwo}) 
  - \log(\|\bar \q - \bar \q_{\revision{\sigma_{l-1}}}\|_{\LtwoLtwo})}{\log(h_l) - \log(h_{l-1})}\\
    (\mathrm{EOC}_k)_l &:= \dfrac{\log(\|\bar \q - \bar \q_{\revision{\sigma_l}}\|_{\LtwoLtwo}) 
    - \log(\|\bar \q - \bar \q_{\revision{\sigma_{l-1}}}\|_{\LtwoLtwo})}{\log(k_l) - \log(k_{l-1})}
\end{align*}
\Cref{subfig:Ex1_hConvergence} displays convergence with respect to the spacial 
discretization parameter for fixed $k$. The theoretical convergence order of 1 can be observed.
\Cref{subfig:Ex1_kConvergence} depicts the convergence results for the time discretization parameter
for different values of $h$. Note that for this comparison, we have choosen discretization levels in time,
such that the two  boundary of the active set $[1/4,3/4]$ are midpoints of the discrete subintervals,
in order to exclude superconvergence effects.
Due to the structure of the discretization, the number of degrees of freedom,
grows rather quickly, which is why very fine discretizations are expensive.
\Cref{subfig:Ex1_kConvergence} shows, that if the spacial discretization parameter is chosen too \revC{large}, 
the stagnation phase sets in rather early. If one only observes the coarse discretizations,
the observed order of convergence is skewed, which led to an estimated order of convergence of
0.85 reported in \cite{christof_new_2021}.
As the derivation of the present example with an analytic reference solution requires the 
involved functions to have some precise regularity, we next propose an example with 
desired state that has low regularity in time. In this case a numerical reference solution is 
needed to measure the errors.

\begin{figure}[h]
  \begin{minipage}[t]{0.45\textwidth}
    \small
    \centering
  \includegraphics[height=5.0cm]{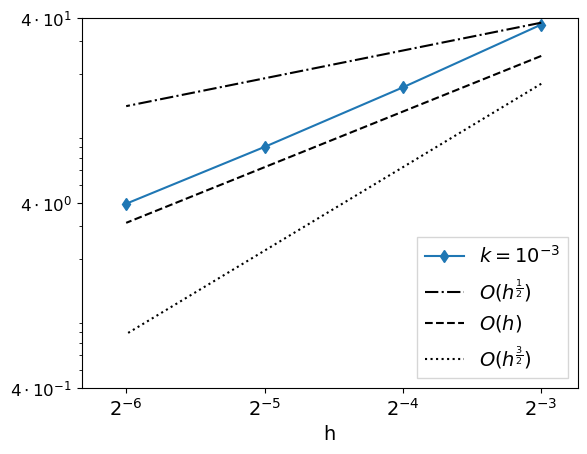}
  \begin{tabular}{ccc}
    \toprule
 & \multicolumn{2}{c}{$k=10^{-3}$} \\
 \cmidrule(lr){2-3}
\multicolumn{1}{c}{$h$} & err & $\mathrm{EOC}_h$ \\
\midrule
$ 2^{-3}$ & 36.954 & - \\
$ 2^{-4}$ & 16.899 & 1.12 \\
$ 2^{-5}$ & 8.0461 & 1.07 \\
$ 2^{-6}$ & 3.9657 & 1.02 \\
\bottomrule
  \end{tabular}
  \caption{Convergence with respect to $h$ for $k=10^{-3}$.}\label{subfig:Ex1_hConvergence}
\end{minipage}
\begin{minipage}[t]{0.55\textwidth}
  \small
  \centering
  \includegraphics[height=5.0cm]{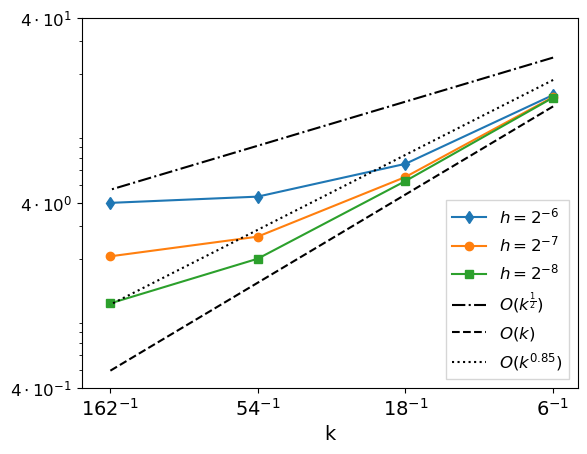}
  \begin{tabular}{rcccccc}
    \toprule
 & \multicolumn{2}{c}{$h=2^{-6}$} & \multicolumn{2}{c}{$h=2^{-7}$} & \multicolumn{2}{c}{$h=2^{-8}$} \\
\cmidrule(lr){2-3} \cmidrule(lr){4-5} \cmidrule(lr){6-7}
\multicolumn{1}{c}{$k$} & err & $\mathrm{EOC}_k$ & err & $\mathrm{EOC}_k$ & err & $\mathrm{EOC}_k$\\
 \midrule
$6^{-1}$    & 15.373 & - & 14.977 & - & 14.878 & -\\
$18^{-1}$   & 6.5178 & 0.78 & 5.5318 & 0.90 & 5.2592 & 0.94\\
$54^{-1}$   & 4.3312 & 0.37 & 2.6310 & 0.67 & 1.9980 & 0.88\\
$162^{-1}$  & 4.0066 & 0.07 & 2.0579 & 0.22 & 1.1443 & 0.50\\

\bottomrule
  \end{tabular}
  \caption{Convergence with respect to $k$ for different values of $h$.}
  \label{subfig:Ex1_kConvergence}
\end{minipage}
\caption{Numerical observation of the error $\|\bar \q - \bar \q_\sigma\|_{\LtwoLtwo}$ for Example 1.}
\end{figure}

\subsection{Example 2}~\\
In this example we consider an optimal control problem where we specify rough data 
$\uu_d$, and $\ww$, while keeping the domain $\Omega = (0,1)^2$, $I = (0,1]$ and 
$\q_a = (-\infty,-\infty)^T$, $\q_b = (+\infty,+\infty)^T$. We specify
\begin{equation*}
  \uu_d = 5 \cdot 10^{4} \varphi(t)
  \begin{pmatrix}
    \sin(\pi x_2) \cos(\pi x_2) \sin(\pi x_1)^2\\
    -\sin(\pi x_1) \cos(\pi x_1) \sin(\pi x_2)^2
  \end{pmatrix}
  \quad \text{where}
  \quad 
  \varphi(t) = \begin{cases}{}
    \sqrt{t-\frac{1}{5}}\cdot (\frac{2}{5}-t) & \text{if } t \in [ \frac{1}{5} , \frac{2}{5} ],\\
    -\sqrt{t-\frac{3}{5}}\cdot (\frac{4}{5}-t) & \text{if } t \in [ \frac{3}{5} , \frac{4}{5} ],\\
    0 & \text{else.}
  \end{cases}
\end{equation*}
The weight  in the state constraint is given by 
\begin{equation*}
    \ww = \chi_{\{(x_1-0.5)^2+(x_2-0.5)^2 \le 0.125\}}(x_1,x_2) \cdot
    \begin{pmatrix}
        x_2-0.5\\
        -x_1+0.5
    \end{pmatrix}
\end{equation*}
and the scalar constraint is given by $\beta\equiv 1$.
Moreover, we consider the regularization parameter $\alpha = 10^{-4}$.
As in this case, no analytical optimal solution is known, we estimate the errors using a numerical reference
solution on a fine grid. To this end, we discretize the problem on 960 time intervals and 128 subdivisions
of $\Omega$ in each direction. 
\revB{%
  Note that due to this evaluation of the errors, we expect a faster convergence than theoretically derived,
  as on the finest discretization level, the error would equal to $0$.
}
In \Cref{subfig:Ex2_hConvergence} we can again observe order 1 convergence with respect to $h$ for 
fixed time discretization.
In \Cref{subfig:Ex2_kConvergence} we observe that the convergence with respect to $k$ exhibits a rate of 
about $0.6$, which is much closer to the analytically derived $0.5$ than the rate observed for the smooth
example 1.

\begin{figure}[H]
  \begin{minipage}[b]{0.35\textwidth}
  \includegraphics[height=0.8\textwidth]{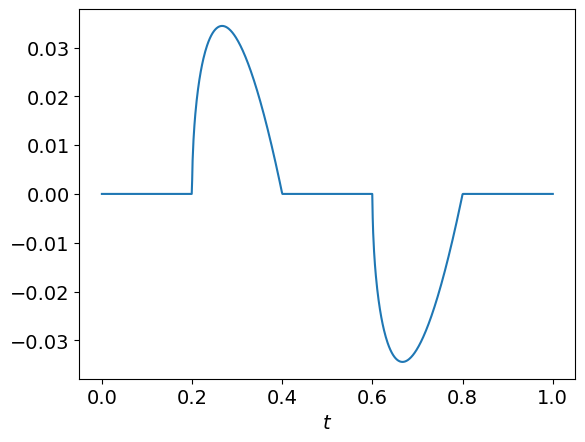}
  \caption{Time function $\varphi(t)$.}
\end{minipage}
\hspace{2pt}
  \begin{minipage}[b]{0.28\textwidth}
  \includegraphics[height=\textwidth]{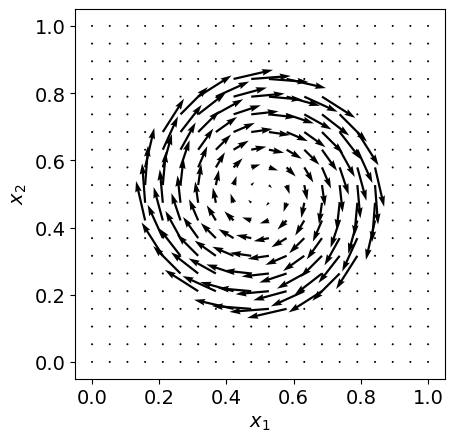}
  \caption{Weight function $\ww$.}
\end{minipage}
\begin{minipage}[b]{0.35\textwidth}
  \includegraphics[height=0.8\textwidth]{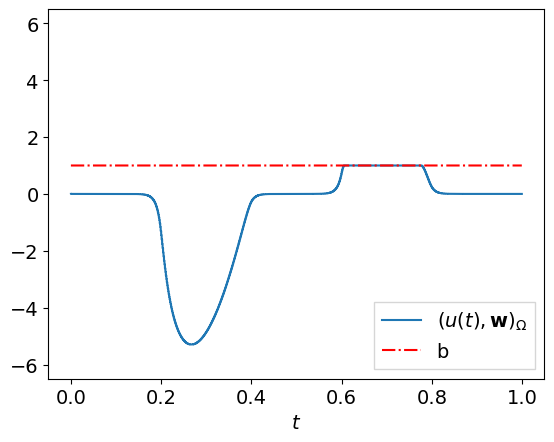}
  \caption{Satisfaction of the state constraint.}
\end{minipage}
\caption{Data for Example 2.}
\end{figure}

\begin{figure}[H]
\begin{minipage}[t]{.45\textwidth}
    \begin{center}
        \includegraphics[height=5cm]{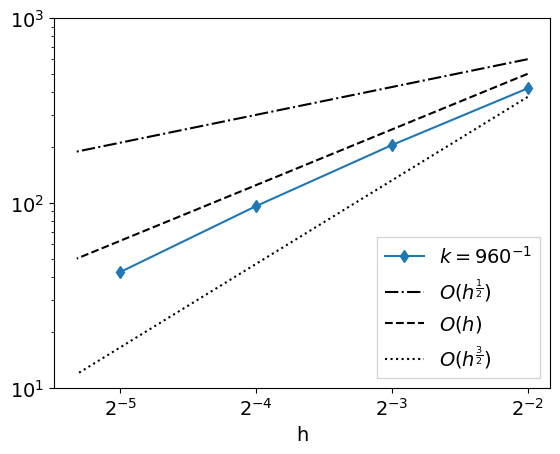}
        \small
        \begin{tabular}{ccc}
          \toprule
          $h$     & err & $\mathrm{EOC}_h$\\
          \midrule
          $2^{-2}$ & 417.33 & - \\
          $2^{-3}$ & 205.87 & 1.01\\       
          $2^{-4}$ & 96.306 & 1.09\\
          $2^{-5}$ & 42.372 & 1.18\\
          \bottomrule
        \end{tabular}
        \vspace{4mm}
    \end{center}
    \caption{Convergence with respect to $h$ for $k=960^{-1}$.}
        \label{subfig:Ex2_hConvergence}
    \end{minipage}
    \begin{minipage}[t]{0.45\textwidth}
        \begin{center}
  \includegraphics[height=5cm]{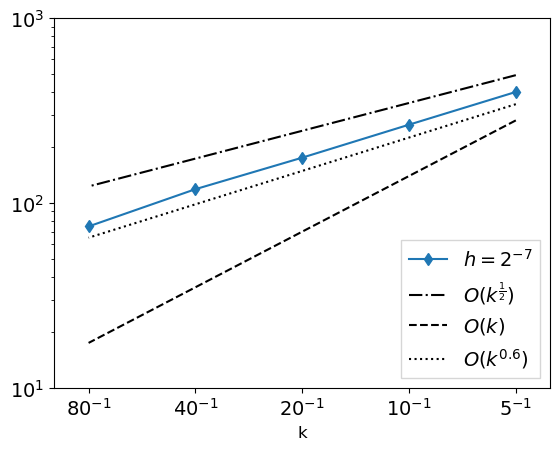}
  \small
  \begin{tabular}{ccc}
    \toprule
    $k$     & err & $\mathrm{EOC}_k$\\
    \midrule
    $5^{-1}$   & 399.06 & - \\
    $10^{-1}$  & 265.39 & 0.58\\       
    $20^{-1}$  & 175.97 & 0.59\\
    $40^{-1}$  & 118.80 & 0.56\\
    $80^{-1}$  & 74.697 & 0.66\\
    \bottomrule
  \end{tabular}
  \end{center}
  \caption{Convergence with respect to $k$ for $h=2^{-7}$.}
  \label{subfig:Ex2_kConvergence}
\end{minipage}
\caption{Numerical observation of the error $\|\bar \q - \bar \q_\sigma\|_{\LtwoLtwo}$ for Example 2.}
\end{figure}

\revision{\subsection{Example 3}
  In the previous examples, no control constraints were present. To highlight, that the derived error estimates
  are indeed not influenced by the control constraints, we augment the Example of \Cref{subsec:examp_1}
  by a control constraint 
  $\bar \q \le 200\cdot \mathds{1}$. The remaining choices of $\Om, \ww, \uu_d, \alpha$ and $\beta$ are 
  kept the same. Due to the presence of the control constraint, an analytic solution $\bar \q$ to the optimal
  control problem is not known, and we again compare to a fine-grid solution, computed with 
  $k=960^{-1}$ and $h=2^{-7}$.
  The numerically observed orders of convergence in $k$ and $h$, as well as the active set of the control 
  contstraint for a fixed poin in time can be observed in \Cref{fig:control_Contraint_example}.
  It has to be noted, that while the discrete problem has a similar structure to the one without control
  constraint, and can be solved by the same PDAS algorithm, the resulting saddle point systems can be much larger
  for large active sets. Efficient preconditioners for such problems are needed, to study the performance of 
  this algorithm for finer discretizations.
\begin{figure}[h]
  \begin{center}
  \begin{minipage}[t]{0.25\textwidth}
    \small
    \centering
  \begin{tabular}{ccc}
    \toprule
 & \multicolumn{2}{c}{$k=960^{-1}$} \\
 \cmidrule(lr){2-3}
\multicolumn{1}{c}{$h$} & err & $\mathrm{EOC}_h$ \\
\midrule
$ 2^{-2}$ & 64.049 & - \\
$ 2^{-3}$ & 35.438 & 0.85\\
$ 2^{-4}$ & 16.024 & 1.14\\
$ 2^{-5}$ & 6.8962 & 1.21\\
\bottomrule
  \end{tabular}
\end{minipage}
\begin{minipage}[t]{0.25\textwidth}
  \small
  \centering
  \begin{tabular}{rcc}
    \toprule
 & \multicolumn{2}{c}{$h=2^{-7}$}\\
\cmidrule(lr){2-3} 
\multicolumn{1}{c}{$k$} & err & $\mathrm{EOC}_k$ \\
 \midrule
$5^{-1}$    & 18.915 & -    \\
$10^{-1}$   & 9.1295 & 1.05\\
$20^{-1}$   & 4.5816 & 0.99\\
$40^{-1}$   & 2.3133 & 0.98\\
$80^{-1}$   & 1.1640 & 0.99\\
$160^{-1}$  & 0.5827 & 0.99\\
\bottomrule
  \end{tabular}
\end{minipage}
\begin{minipage}[t]{0.25\textwidth}
  \vspace{-18mm}
  \includegraphics[width=0.97\linewidth]{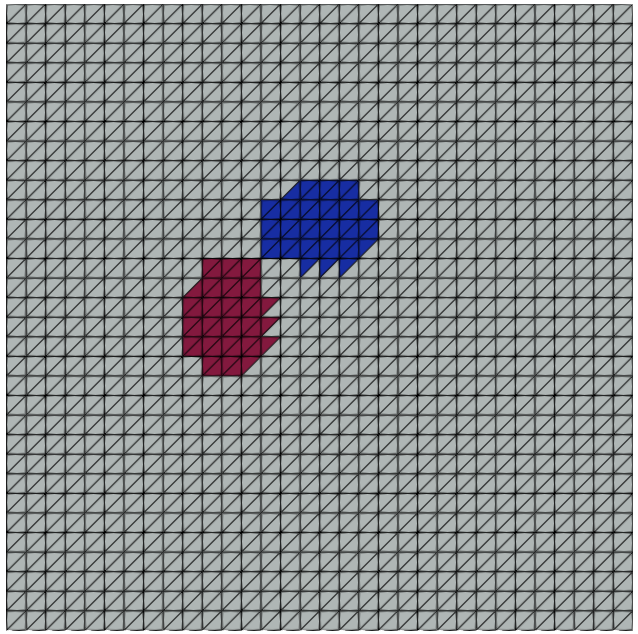}
\end{minipage}
\vspace{2mm}
\caption{Left: Orders of convergence for an example with control constraints.
  Right: Active sets of the discrete optimal control for $h=2^{-5}$, $k=30^{-1}$ at $t=1/3$.
(grey: inactive, blue: constraint in $x_1$ direction active, red: constraint in $x_2$ direction active).}
\label{fig:control_Contraint_example}
\end{center}
\end{figure}
%
}

\bibliographystyle{siam}
\bibliography{stokes_state}

%

\end{document}